\documentclass{amsart}

\usepackage{amsmath,amsthm}
\usepackage{amsfonts,amssymb}
\usepackage{accents}
\usepackage{enumerate}
\usepackage{accents,color}
\usepackage{graphicx}

\usepackage{longtable}
\usepackage{enumerate}
\allowdisplaybreaks

\newcommand{\lvt}{\left|\kern-1.35pt\left|\kern-1.3pt\left|}
\newcommand{\rvt}{\right|\kern-1.3pt\right|\kern-1.35pt\right|}

\newtheorem{thm}{Theorem}[section]

\newtheorem{prop}[thm]{Proposition}

\newtheorem{defn}[thm]{Definition}

\theoremstyle{remark}

 \def\la{{\langle}}
 \def\ra{{\rangle}}

 \def\d{\mathrm{d}}
 \def\e{{\mathrm e}}

 \def\sw{{\mathsf w}}

 \def\sJ{{\mathsf J}}
 \def\sL{{\mathsf L}}
 
 \def\sP{{\mathsf P}}
 \def\sQ{{\mathsf Q}}
 \def\sR{{\mathsf R}}
 \def\sS{{\mathsf S}}
 \def\sV{{\mathsf V}}
 \def\sU{{\mathsf U}}
 \def\sY{{\mathsf Y}}

 \def\a{{\alpha}}
 \def\b{{\beta}}
 \def\g{{\gamma}}

 \def\l{{\lambda}}

 \def\la{{\langle}}
 \def\ra{{\rangle}}
 
 \def\ab{{\mathbf a}}
 \def\bb{{\mathbf b}}
 
 \def\ib{{\mathbf i}}
 \def\jb{{\mathbf j}}
 \def\kb{{\mathbf k}}
 
 \def\mb{{\mathbf m}}

 \def\CV{{\mathcal V}}

 \def\BB{{\mathbb B}}

 \def\NN{{\mathbb N}}
 \def\PP{{\mathbb P}}
 \def\QQ{{\mathbb Q}}
 \def\RR{{\mathbb R}}
 
 \def\VV{{\mathbb V}}
 \def\TT{{\mathbb T}}

 \def\one{\mathbf{1}}
\def\lla{\langle{\kern-2.5pt}\langle} 
\def\rra{\rangle{\kern-2.5pt}\rangle}

\newcommand{\wh}{\widehat}

\def\f{\frac}

\graphicspath{{./}}
\begin{document}
 
\title{Monomial and Rodrigues orthogonal polynomials on the cone}

\author{Rab\.{I}a Akta\c{S}}
\address{Ankara University, Faculty of Science, Department of Mathematics, 06100, Tando\u{g}an Ankara, Turkey}
\email{raktas@science.ankara.edu.tr}
\author{Am\'ilcar Branquinho}
\address{Departamento de Matem\'atica, Universidade de Coimbra, 3001-454 Coimbra, Portugal}
\email{ajplb@mat.uc.pt}
\author{Ana Foulqui\'e-Moreno}
\address{Departamento de Matem\'atica, Universidade de Aveiro, 3810-193 Aveiro, Portugal}
\email{foulquie@ua.pt}
\author{Yuan~Xu}
\address{Department of Mathematics, University of Oregon, Eugene, 
OR 97403--1222, USA}
\email{yuan@uoregon.edu} 
\thanks{The first author was partially supported by TUBITAK Research Grant Project \#120F140.
The second author acknowledges Centro de Matem\'atica da Universidade de Coimbra (CMUC)
-- UID/MAT/00324/2020, funded by the Portuguese Government through FCT/MEC and co-funded by the 
European Regional Development Fund through the Partnership Agreement PT2020. 
The third author acknowledges CIDMA Center for Research and Development in Mathematics and Applications
(University of Aveiro) and the Portuguese Foundation for Science and Technology (FCT) within the
project UIDB/MAT/UID/04106/2020 and UIDP/MAT/04106/2020. The fourth author was partially supported by 
Simons Foundation Grant \#849676 and was grateful for an award from the Alexander von Humboldt Foundation.}
\date{\today} 
\subjclass[2010]{33C50, 41A10, 41A63}
\keywords{Orthogonal polynomials, cones, monomial polynomials, Rodrigues formula, Laguerre and Jacobi}
 
\begin{abstract} 
We study two families of orthogonal polynomials with respect to the weight function $w(t)(t^2-\|x\|^2)^{\mu-\f12}$,
$\mu > -\f 12$, on the cone $\{(x,t): \|x\| \le t, \, x \in \mathbb{R}^d, t >0\}$ in $\mathbb{R}^{d+1}$. The first family 
consists of monomial polynomials $\mathsf{V}_{\kb,n}(x,t) = t^{n-|\mathbf{k}|} x^\mathbf{k} + \cdots$ for 
$\mathbf{k} \in \mathbb{N}_0^d$ with $|\mathbf{k}| \le n$, which has the least $L^2$ norm among all polynomials 
of the form $t^{n-|\kb|} x^\kb + \sP$ with $\deg \sP \le n-1$, and we will provide an explicit construction for 
$\mathsf{V}_{\kb,n}$. The second family consists of orthogonal polynomials defined by the Rodrigues type 
formulas when $w$ is either the Laguerre weight or the Jacobi weight, which satisfies a generating function in both cases. 
The two families of polynomials are partially biorthogonal.
\end{abstract} 
 
\maketitle 

\section{Introduction}
\setcounter{equation}{0}

One of the fundamental differences between orthogonal polynomials in one variable and several variables
lies in the multiplicity of orthogonal bases. Let $L^2(\Omega, W)$ be a weighted $L^2$ space on the domain
$\Omega$ in $\RR^d$. For $n =0,1,2,\ldots$, let $\CV_n^d$ denote the space of orthogonal polynomials of 
degree $n$ in $L^2(\Omega, W)$, which consists of polynomials of degree $n$ that are orthogonal to all
polynomials of lower degrees for the inner product of $L^2(\Omega, W)$. For $d =1$, the space 
$\CV_n^d$ is one dimension, so the orthogonal polynomial of degree $n$ is unique up to a constant 
multiple. For $d >1$, the space $\CV_n^d$ has the dimension $\binom{n+d-1}{n}$ and it has infinitely many
distinct bases. Different bases are often needed for different occasions or applications, and each basis may 
catch a specific perspective of the orthogonal structure. One could choose an orthonormal basis and there 
are infinitely many of those. There are also outstanding bases that may not be mutually orthogonal. 

Among the plethora of choices, we single out two that are not orthogonal but are biorthogonal with each other. 
The first basis consists of monomial polynomials
$$
 V_\kb(x) = x^{\kb} + R_\kb(x), \qquad \kb \in \NN_0^d, \quad R_\kb \in \Pi_{n-1}^d, 
$$
where $\Pi_n^d$ denotes the space of polynomials of degree $n$ in $d$ variables, which is uniquely 
determined by the requirement that $\{V_{\kb}: |\kb| = n\}$ is a basis of $\CV_n^d$. This basis 
is not necessarily an orthogonal basis of $\CV_n^d$. One of the important features of this basis is that
$V_\kb$ has the least $L^2$ norm among all polynomials of the form $x^\kb - P$, $P \in \Pi_{n-1}^d$;
more precisely, 
\begin{equation}\label{eq:bestapp}
 \inf_{P\in \Pi_{n-1}^d} \left\| x^\kb - P\right \|_{L^2(\Omega,W)} = \left \|V_\kb \right \|_{L^2(\Omega,W)}.
\end{equation}
The second basis consists of polynomials $U_\kb$, $|\kb| = n$, defined by the Rodrigues type formulas that 
involve repeated differentiation, which can only be defined for special weight functions and domains. The 
study of such bases was initiated for classical orthogonal polynomials on the unit disk and the unit ball $\BB^d$ 
by Hermite and his contemporaries (see, for example, \cite{AF}), where the orthogonality is defined for the weight 
function $(1-\|x\|^2)^{\mu-\f12}$, $\mu > -\f 12$, on $\BB^d$. For the unit ball, these two families of polynomials 
can be given explicitly, satisfy elegant generating functions, and turn out to be mutually 
biorthogonal \cite{AF, DX, EMOT} (see Section 2.2 below). 

In the present paper, we study orthogonal polynomials on the conic domain
\begin{align*}
 \VV^{d+1} = \{(x,t) \in \RR^{d+1}: \, \|x\| \le t, \, x \in \RR^d, \, t \ge 0\},
\end{align*}
on $\RR^{d+1}$ and the orthogonality is defined with respect to the weight function 
$$
 W(x,t) = w(t) (t^2 - \|x\|^2)^{\mu - \f12}, \qquad \mu > -\tfrac12, 
$$
where $w$ is a weight function defined on $\RR_+ = [0,\infty)$ and the cone could be regarded as finite if $w$ 
is supported on the finite interval, say, $[0,1]$. Orthogonal polynomials on the cone are studied only recently.
The two most important cases are the Laguerre polynomials on the cone, with $w(t) = t^\b \e^{-t}$, and the 
Jacobi polynomials on the cone, with $w(t) = t^\b(1-t)^\g$. In these cases, an orthonormal basis is defined
in \cite{X20} and used to show that orthogonal polynomials in $\CV_n^{d+1}$ are eigenfunctions of a second-order
differential operator and they satisfy an addition formula. The latter gives a closed-form formula for the 
reproducing kernel of the space $\CV_n^{d+1}$, which provides essential tools for an extensive study on
approximation theory and computational harmonic analysis over the cone \cite{X21a, X21b}. The study of the 
orthogonal structure over the cone is still 
in its beginning, the present paper aims to explore the monomial basis and the basis defined by the Rodrigues formulas.

Our main work lies in understanding the monomial basis. In the variable $(x,t)\in \RR^{d+1}$, such a basis
takes of the form
$$
 \sV_{\kb,n}(x,t) = t^{n-|\kb|} x^\kb + \sP_{\kb,n} \quad\hbox{with}\quad \sP_{\kb,n} \in \Pi_{n-1}^{d+1}.
$$
Given the nature of the weight function, one may expect that the polynomial $\sV_{\kb,n}$ can be constructed
simply in terms of the monomial polynomials $V_\kb$ on the unit ball and orthogonal polynomial of one 
variable in the $t$ variable. The problem, however, turns out to be more subtle, since $V_\kb$ can only appear in the form of $t^{|\kb|} V_{\kb}(\frac{ 
 x}{t})$, which however is a homogeneous polynomial, far from a 
monomial. Even the case $d=1$, or $\VV^2 \subset \RR^2$, turns out to be non-trivial. One of our main results provides 
an explicit formula for the polynomials $\sV_{\kb,n}$ and its $L^2$ norm. For the Laguerre and the Jacobi cases, 
we also define the second family of polynomials, denoted by $\sU_{\kb,n}$, via the Rodrigues type formulas.
The polynomials $\sU_{\kb,n}$ for $|\kb| =n$ is also a basis of $\CV_n^{d+1}$ and they satisfy a generating 
function. The two families, $\{\sV_{\kb,n}\}$ and $\{\sU_{\kb,n}\}$, are, however, only partially biorthogonal, in 
contrast to the case of the unit ball.

The paper is organized as follows. The next section is preliminary, where we recall the results of classical 
orthogonal polynomials on the unit ball and develop orthogonal polynomials on the cone. The monomial 
polynomials on the cone, called the $V$-family, are defined and studied in the third section, where we first 
consider the case $d =1$ to illustrate the idea. The second family of polynomials, the $U$-family, is defined 
via Rodrigues formulas in the fourth section and shown to be partially biorthogonal to the $V$-family and, in 
$d =1$ case, a basis that is biorthogonal to the $V$-family is explicitly constructed in terms of the $U$-family. 
Finally, we discuss generating functions for these polynomials in the fifth~section. 

\section{Preliminary}
\setcounter{equation}{0}

In the first subsection, we discuss orthogonal polynomials, especially monomial orthogonal polynomials, in
general. In the second section, we illustrate the general setup by recalling results from classical orthogonal 
polynomials on the unit ball, which will also be needed in our study. In the last subsection, we lay down 
the basics for orthogonal polynomials on the cone. 

Throughout this paper, we use multi-index notation: for $\kb= (k_1,\ldots, k_d) \in \NN_0^d$ and 
$\ab=(a_1,\ldots,a_d)\in \RR^d$, we shall write, for example, 
$$
 \kb ! = k_1!\cdots k_d!, \qquad (\ab)_\kb = (a_1)_{k_1} \cdots (a_d)_{k_d}, \qquad \sum_{\jb \le \kb} A_\jb 
 = \sum_{j_1=0}^{k_1}\cdots \sum_{j_d=0}^{k_d} A_{\jb},
$$
where $(a)_k = a(a+1)\cdots (a+k-1)$ denotes the Pochhammer symbol. 

\subsection{Monomial orthogonal polynomials} 
Let $\sw$ be a weight function on a domain $\Omega \subset \RR^d$. Let $\la \cdot,\cdot \ra_{\sw}$ be the inner product defined by 
$$
 \la f, g\ra_{\sw} = \int_\Omega f(x) g(x) \sw(x) \d x.
$$
Let $\Pi_n^d$ denote the space of polynomials of degree at most $n$ in $d$ variables. A polynomial~$P$ of degree $n$ 
is an orthogonal polynomial with respect to the inner product~if 
$$
 \la P, Q\ra_{\sw} = 0, \qquad \forall Q \in \Pi_{n-1}^d.
$$
Assume that $\sw$ is regular so that orthogonal polynomials with respect to the inner product are well-defined. 
Let $\CV_n^d(\Omega, \sw)$ be the space of orthogonal polynomials of degree $n$ in $d$ variables with respect 
to the inner product. It is known that 
$$
 r_n^d:= \dim \CV_n^d(\Omega, \sw) = \binom{n+d-1}{n}, \qquad n =0,1,2,\ldots. 
$$
For $d > 1$, the space $\CV_n^d(\Omega, \sw)$ contains infinitely many
bases. Moreover, since orthogonality is defined as orthogonal to all polynomials of lower degrees, elements of a 
basis may not be mutually orthogonal. A basis $\{P_j^n: 1 \le j \le r_n^d\}$ of $\CV_n^d(\Omega, \sw)$ 
is called an orthogonal basis if $\la P_j^n, P_k^n\ra_\sw = 0$ whenever $j \ne k$ and it is called an orthonormal 
basis if, in addition, $\la P_j^n, P_j^n\ra_\sw =1$ for all $j$. It is often convenient to consider an orthonormal basis 
for many applications, such as dealing with the Fourier orthogonal series. Notice, however, that if 
$\PP_n =\{P_j^n: 1 \le j \le r_n^d\}$ is an orthonormal basis, then so is $Q \PP_n$ if $\PP_n$ is regarded as 
a column vector and $Q$ is an orthogonal matrix of size $\binom{n+d-1}{n}$. Hence, there are infinitely
many distinct orthonormal bases. Our goal in this paper is to consider the monomial basis 
$\{V_\kb: |\kb| = n, \kb \in \NN_0^d\}$, where $V_\kb \in \CV_n^d(\Omega, \sw)$ is characterized by 
$$
 V_\kb (x) = x^\kb + P_\kb, \qquad P_\kb \in \Pi_{n-1}, \quad n = |\kb|;
$$
in other words, the polynomial $V_\kb$ of degree $n$ contains a single monomial, $x^\kb$, of degree $n$. 
The monomial orthogonal polynomial $V_\kb$ is the error of the least square approximation to the monomial $x^\kb$
from $\Pi_{|\kb| -1}^d$; more precisely, it satisfies \eqref{eq:bestapp} since $V_\kb = x^\kb + P_{\kb}$ is 
orthogonal to all polynomials in $\Pi_{n-1}^d$. Moreover, in the case of several families of classical orthogonal polynomials, 
such as those on the unit ball or the standard simplex, the monomial orthogonal polynomials possess fine 
structures and are biorthogonal to another family of orthogonal polynomials defined by the Rodrigues type 
formulas. We illustrate this in the next subsection. 

\subsection{Orthogonal polynomials on the unit ball}
For $\mu > -\f12$, the classical weight function on the unit ball $\BB^d = \{x\in \RR^d: \|x\| \le 1\}$ of $\RR$ is 
\begin{align*}
 \sw_\mu(x) = (1-\|x\|^2)^{\mu - \f12}, \quad \mu > -\tfrac12, \quad x\in \BB^d,
\end{align*}
the normalization constant $b_\mu^\BB$, so that $b_\mu^\BB \sw_\mu(x)$ has the unit integral, is given by 
$$
 b_\mu^\BB = \frac{1}{\int_{\BB^d} \sw_\mu(x) \d x} = \frac{\Gamma(\mu+\frac{d+1}{2})}{\pi^{\frac d 2}\Gamma(\mu+\frac{1}{2})}.
$$

When $d =1$, the ball is $[-1,1]$ and the associated orthogonal polynomials are the classical Gegenbauer 
polynomials $C_n^\mu$, which satisfy the orthogonal relation 
\begin{equation} \label{eq:Gegen}
 c_\mu \int_{-1}^1 C_n^\mu(t)C_m^\mu(t)(1-t^2)^{\mu-\f12} \d t = h_n^\mu \delta_{n,m}, \qquad 
 h_n^\mu = \frac{\mu (2\mu)_n}{(n+\mu)n!},
\end{equation}
where $c_\mu =\frac{ \Gamma(\mu+1)}{\sqrt{\pi} \Gamma(\mu+\f12)}$ is the same as $b_\mu^{\BB}$ with $d =1$
and $\mu \ne 0$. If $\mu = 0$, the polynomial $C_n^\mu$ becomes the Chebyshev polynomial $T_n$ under the limit
$$
\lim_{\mu \to 0} \frac 1 \mu C_n^\mu(x) = \frac{2}{n}T_n(x), \qquad n \ge 1. 
$$

For $d > 1$, the orthogonal polynomials with respect to $\sw_\mu$ on $\BB^d$ are well studied; see~\cite[Section 5.2]{DX}. Let $\CV_n(\BB^d, \sw_\mu)$ be the space of OPs for this weight function. There are 
many distinct bases for this space and two orthonormal bases can be given explicitly in either Cartesian or polar 
coordinates. We are mostly interested in the monomial basis and the basis defined by the Rodrigues formula. While
the former has been discussed before, the latter is an extension of the Rodrigues formula for classical orthogonal
polynomials in one variable to several variables; see, for instance,
\cite{AF, AFPP, DX, EMOT, IZ, X05b}, for some examples.
For the unit ball, both bases can be traced back to the work of Hermite \cite{AF}. Following the classical presentation (cf. \cite[Chapter~12]{EMOT} and \cite[Chapter 5]{DX}), we denote the elements of these 
two bases by $V_\kb$ and $U_\kb$ with $\kb \in \NN_0^d$, respectively. 

We consider the $U$-basis given by the Rodrigues formula first. For each $\kb \in \NN_0^d$, the polynomials
$U_\kb$ is defined by 
\begin{equation}\label{eq:RodrigB}
 U_{\kb}(x) = \frac{(-1)^{|\kb|} (2\mu)_{|\kb|}}{2^{|\kb|}(\mu + \f12)_{|\kb|}\kb !}
 \frac{1}{(1-\|x\|^2)^{\mu-\f1 2}} \frac{\partial^{|\kb|}}{\partial_1^{\kb_1} \cdots \partial_d^{\kb_d} }
 (1- \|x\|^2)^{|\kb|+\mu-\f12}. 
\end{equation}
Then $\{U_\kb: |\kb| = n\}$ is a basis of $\CV_n(\BB^d,\sw_\mu)$. These polynomials also satisfy a 
generating function given by
\begin{equation}\label{eq:generatingB}
 \frac{1}{\left((1- \la \ab ,x \ra)^2 + \|\ab \|^2 (1-\|x\|^2)\right)^{\mu}} = \sum_{\kb \in \NN_0^d}
 U_\kb(x) \ab^\kb, \quad \ab \in \BB^d.
\end{equation}
The $V$-basis is the monomial basis, For each $\kb \in \NN_0^d$, the polynomial $V_\kb$ is defined by 
\begin{equation} \label{eq:V-monic1}
 V_{\kb}(x) = x^\kb F_B\left(-\frac{\kb}{2}, \frac{\one -\kb}{2}; 
 -|\kb|- \mu - \frac{d-3}{2}; \frac{1}{x_1^2},\ldots,\frac{1}{x_d^2}\right),
\end{equation}
where $\one = (1,\ldots,1)\in \NN^d$ and $F_B$ is the Lauricella hypergeometric series of $d$ variables 
defined by
\begin{align*}
 F_B(\ab, \bb; c; x) = \sum_{\mb \in \NN_0^d} \frac{(\ab)_\mb (\bb)_\mb}{(c)_{|\mb|} \mb!} x^\mb, \qquad 
 \max_{1\le j \le d} |x_j| < 1.
\end{align*}
It is easy to see that $V_\kb(x) = x^\kb + Q_\kb$ with $Q_\kb \in \Pi_{n-1}^d$, so that $V_\kb$ is monomial for 
$\kb \in \NN_0^d$. Moreover, $\{V_\kb: |\kb|=n\}$ is a basis for $\CV_{|\kb|} (\BB^d,\sw_\mu)$. These polynomials 
satisfy a generating function given by
\begin{equation}\label{eq:generatingBV}
 \frac{1}{\left(1-2 \la \ab ,x \ra + \|\ab \|^2\right)^{\mu+\f{d-1}2}} = \sum_{\kb \in \NN_0^d}
 \frac{2^{|\kb|}}{\kb!} \left(\mu+ \frac{d-1}{2}\right)_{|\kb|} V_\kb(x) \ab^\kb, \quad \ab \in \BB^d.
\end{equation}
As we mentioned before, the polynomial $V_\kb$ has the least $L^2$ norm among all polynomials of the form
$x^\kb - P(x)$, $P\in \Pi_{n-1}^d$. Its norm is computed in \cite[Corollary~4.8]{X05}, which shows in particular 
that, for $\kb \in \NN_0^d$ and $|\kb| =n$, 
\begin{equation} \label{eq:leastBall}
 \inf_{P\in \Pi_{n-1}^d} \left \| x^\kb - P \right\|_{L^2(\BB^d, \sw_\mu)}^2
 = \frac{\l \kb!}{2^{n-1}(\l)_n} \int_0^1 \prod_{i=1}^d P_{k_i}(t) t^{n+2 \l -1} \d t,
\end{equation}
where $\l = \mu + \frac{d-1}{2} > 0$ and $P_n$ is the Legendre polynomial of degree $n$, which implies,
in particular, that the integral on the right-hand side is positive, a fact that is not obvious. For $d =2$, 
a different expression for this quantity is given in \cite{BHN}. For small $n$, the least square of some
symmetric monomials are given in \cite{AY}, which can be deduced from the norm of $V_{\kb}$, see \cite{X05}.
Orthogonal polynomials are also used in deriving the exact constant in the Bernstein inequality on the ball in the recent work of \cite{K}.

Neither $U$-basis nor $V$-basis are orthogonal bases for $\CV_n(\BB^d, \sw_\mu)$. They are, however, biorthogonal
with respect to each other. More precisely, they satisfy
$$
 b_\mu^\BB \int_{\BB^d} U_\kb(x) V_\jb(x) \sw_\mu(x) \d x 
 =\f{ (2\mu)_{|\kb|}}{ 2^{|\kb|}(\mu+\f{d+1}2)_{|\kb|}}\delta_{\jb,\kb}.
$$

We will need the explicit expression of $V_\kb(x)$ in monomials, which is derived by rewriting \eqref{eq:V-monic1}
in the following form
\begin{equation} \label{eq:V-monic}
V_{\kb}(x) = \sum_{2 \jb \leq \kb} b_{\kb,\jb}^\mu x^{\kb- 2 \jb}, \quad \hbox{where}\quad
 b_{\kb,\jb}^\mu:=\frac{(\frac{-\kb}{2})_{\jb} (\frac{-\kb+1}{2})_{\jb}}{(-|\kb|-\lambda +1)_{|\jb|} \jb!},
\end{equation}
and we shall denote, throughout the rest of the paper, 
$$
 \l = \mu + \frac{d-1}{2}.
$$
The relation \eqref{eq:V-monic} can be reversed, as seen in \cite[Corollary 6.2]{X05}, as 
\begin{equation} \label{eq:monic-V}
 x^{\kb} = \sum_{2 \jb \leq \kb} c_{\kb, \jb}^\mu V_{\kb-2 \jb} (x), 
 \end{equation}
where the coefficients are given by
\begin{equation} \label{eq:monic-ckj}
 c_{\kb, \jb}^\mu = \frac{(-1)^{|\jb|} (-\frac{\kb}{2})_{\jb} (\frac{-\kb+\one}{2})_{\jb}}{(-|\kb| -\lambda+1)_{2|\jb|}\, \jb!}
 \left(2 (-|\kb| -\lambda+1)_{|\jb|} - (-|\kb| - \lambda)_{|\jb|} \right).
 \end{equation} 
Putting the two identities together, we obtain 
\begin{align*}
 x^{\kb} & = \sum_{2 \ib \leq \kb} c_{\kb,\ib}^\mu V_{\kb-2 \ib} (x)= \sum_{2 \ib \leq \kb} c_{\kb,\ib}^\mu 
 \sum_{2 \jb \leq \kb - 2 \ib} b_{\kb-2\ib, \jb}^\mu x^{\kb - 2 \ib-2 \jb}
 \\
 & = \sum_{2 \ib \leq \kb} c_{\kb,\ib}^\mu \sum_{2 \jb \leq \kb} b_{\kb-2\ib, \jb-\ib}^\mu x^{\kb-2 \jb}
 \\
 & = \sum_{2 \jb \leq \kb} x^{\kb-2 \jb} \sum_{\ib \leq \jb} c_{\kb,\ib}^\mu b_{\kb-2\ib, \jb-\ib}^\mu 
\end{align*}
Consequently, it follows readily that 
\begin{equation} \label{eq:bc=delta}
 \sum_{\ib \leq \jb} b_{\kb, \ib}^\mu c_{\kb - 2 \ib, \jb-\ib}^\mu = \delta_{\jb,0}, \quad j \in \mathbb N_0.
\end{equation}
This identity will be used in the sequel. 
 
\subsection{Orthogonal polynomials on the cone}
Let $w$ be a weight function on an interval in $\RR$. We can assume that the interval is either $[0, 1]$ or 
$\RR_+ = [0,\infty)$ without losing generality. On the cone $\VV_0^{d+1}$, we define 
\begin{equation}\label{eq:Wmu-cone}
 W_{\mu}(x,t) = w(t) (t^2-\|x\|^2)^{\mu-\f12}, \quad \mu > -\tfrac12 ,
\end{equation}
and define the inner product 
$$
 \la f,g\ra_{\mu}: = b_\mu \int_{\VV^{d+1}} f(x,t) g(x,t) W_{\mu}(x,t) \d x \d t,
$$
where $b_\mu$ is a normalization constant so that $\la 1,1\ra_{\mu} = 1$ and 
$$
 b_\mu = b_\mu^\BB \times b_{\mu}^w \quad\hbox{with}\quad b_{\mu}^w = \frac{1}{\int_0^\infty t^{d+2\mu-1} w(t) \d t }.
$$
This can be easily verified by using the separation of variables 
$$
 \int_{\VV^{d+1}} f(x,t) W_{\mu}(x,t) \d x \d t= \int_0^\infty \int_{\BB^d} f(t y, t)(1-\|y\|^2)^{\mu - \f12} \d y 
 \, t^{d+2\mu-1} w(t) \d t.
$$
Let $\CV_n(\VV^{d+1}, W_{\mu})$ be the space of orthogonal polynomials of degree $n$ in $d+1$ variables 
with respect to the inner product $\la f,g\ra_{\mu}$. A basis of this space can be given in terms of orthogonal
polynomials on the unit ball and a family of orthogonal polynomials in one variable. 

\begin{prop} \label{prop:OPcone}
Let $\PP_m = \{P_\kb: |\kb| = m\}$ be a basis of $\CV_m(\BB^d, \sw_\mu)$ for $m \le n$ and let $q_{n-m}^\a$ 
be an orthogonal polynomial in one variable with respect to the weight function $t^\a w(t)$ on $\RR_+$. Define \footnote{Throughout this paper, we will adopt the convention that all orthogonal polynomials on the cone will be
denoted by letters in the sans serif font, such as $\sQ_{\kb}$, $\sS_{\kb}$, and $\sV_{\kb}$. }
\begin{equation} \label{eq:sQcone}
 \sQ_{\kb,n} (x,t) = q_{n-m}^{\a_m}(t) t^m P_\kb\!\left(\frac{x}{t}\right), \qquad |\kb| =m, \quad 0 \le m \le n,
\end{equation}
where $\a_m = d+2m+2 \mu-1$. Then $\QQ_n = \{ \sQ_{\kb,n}: |\kb| =m, \, 0 \le m \le n\}$ is a basis of $\CV_n(\VV^{d+1}, W_\mu)$. 
In particular, if $\PP_m$ is an orthogonal basis for $\CV_m(\BB^d, \sw_\mu)$ for $0 \le m \le n$, then $\QQ_n$ 
is an orthogonal basis for $\CV_n(\VV^{d+1}, W_\mu)$. 
\end{prop}

\begin{proof}
This can be easily seen from 
\begin{align*}
 &
\la \sQ_{\kb,n}, \sQ_{\kb',n'} \ra_\mu 
\\
 & \phantom{olaola} = b_\mu \int_0^\infty t^{d+m+m'+2\mu-1}q_{n-m}^{\a_m}(t) q_{n'-m'}^{\a_{m'}}(t)
 \int_{\BB^d} P_\kb^m(y) P_{\kb'}^{m'}\!(y) \sw_\mu (y) \d y w(t) \d t \\
 & \phantom{olaola} = b_\mu \int_0^\infty \!q_{n-m}^{\a_m}(t) q_{n'-m}^{\a_m}(t)t^{\a_m} w(t) \d t 
 \! \int_{\BB^d} \! P_\kb^m(y) P_{\kb'}^{m}(y) \sw_\mu (y) \d y \, \delta_{m,m'} \\
 & \phantom{olaola}
= {\| q_{n-m}^{\a_m} \|^2} b_\mu^{\mathbb B} \int_{\BB^d} P_\kb^m(y) P_{\kb'}^{m}(y) \sw_\mu (y) \d y\,
 \delta_{m,m'}\delta_{n,n'},
\end{align*} 
where we have used the orthogonality of $P_\kb^m$. Moreover, if $\PP_m$ is an orthogonal basis,
then the last integral can be replaced by a constant multiple of $\delta_{\kb,\kb'}$ so that the basis~$\QQ_n$
is an orthogonal basis on the cone. 
\end{proof}

What we are interested in is the monomial basis for $\CV_n(\VV^{d+1}, W_\mu)$. Monomials of degree $n$
in $(x,t)$ variables are $t^{n- |\kb|} x^\kb$ for $\kb \in \NN_0^d$ and $0 \le |\kb| \le n$. We assume that the
monomial orthogonal polynomial $\sV_{\kb,n}$ is of the form 
\begin{equation} \label{eq:Vcone-def}
 \sV_{\kb,n}(x,t) = t^{n-|\kb|} x^\kb + \sR_{\kb}(x,t), \qquad \sR_\kb \in \Pi_{n-1}^d.
\end{equation}
Given the Proposition \ref{prop:OPcone}, it is tempting to consider orthogonal polynomials 
\begin{align} \label{eq:S-OPcone}
 \sS_{\kb,n}(x,t):= q_{n- m}^{\a_{m}}(t) t^{m} V_\kb \left(\f{x}{t} \right), \quad |\kb| =m, \,\, \kb\in \NN_0^d, \,\,
 0\le m \le n,
\end{align}
where $V_\kb$ denotes the monomial orthogonal polynomials on the unit ball, defined in~\eqref{eq:V-monic1},
and $q_{n-m}^{\a_m}$ is the monic orthogonal polynomial in the $t$ variable. However, although both 
$q_{n-m}^{\a_m}$ and $V_\kb$ are monomial polynomials, the polynomials $\sS_{\kb,n}$ are monomial only when
$\kb = 0$ since the polynomial $t^{|\kb|} V_\kb \left(\f{x}{t} \right)$ is a homogeneous polynomial of degree $m$
in the $(x,t)$ variables when $|\kb| > 0$, hence not monomial. The polynomials~$\sS_{\kb,n}$ nevertheless play 
an important role in our construction of the monomial basis in the next section. 

For our study, the most interesting cases are when $w$ is a classical weight function, for which $q_{n-m}^{\a_m}$
can be given explicitly in classical orthogonal polynomials. This comes down to two families of classical weight 
functions, which lead to the Laguerre polynomials on the cone and the Jacobi polynomials on the cone studied 
in \cite{X20}. We review these two families of polynomials below. 

\subsubsection{Laguerre polynomials on the cone}
In this case, the cone is unbounded, 
$$
 \VV^{d+1} = \left \{(x,t): \|x\| \le t, \,\, x\in \RR^d, \,\, 0 \le t < \infty \right\} ,
$$
and the weight function $w(t) = t^\b \e^{-t}$, so that $W_\mu$ in \eqref{eq:Wmu-cone} becomes 
\begin{equation*}
W_{\b,\mu}(x,t) = (t^2-\|x\|^2)^{\mu-\f12} t^\b \e^{-t}, \quad \mu > -\tfrac12, \quad \b > -d,
\end{equation*}
and its normalization constant becomes $b_{\b,\mu}^L$ given by 
\begin{equation*}
 b_{\b,\mu}^L = \frac{1}{\Gamma(2\mu + \b+d)} b_\mu^\BB. 
\end{equation*}

Recall that the Laguerre polynomial $L_n^\a$ is defined by, for $\a > -1$, 
$$
 L_n^\a(t) = \frac{(\a+1)_n}{n!} {}_1F_1(-n; \a+1; t) = \frac{(\a+1)_n}{n!} \sum_{k=0}^n \frac{(-n)_k}{(\a+1)_k k!} t^k ,
$$
and it satisfies the orthogonal relation 
$$
 \frac{1}{\Gamma(\a+1)} \int_0^\infty L_n^\a(t) L_m^\a(t) t^\a \e^{-t} \d t = \frac{(\a+1)_n}{n!} \delta_{m,n}.
$$
The orthogonal polynomials $\sQ_{\kb,n}$ in \eqref{eq:sQcone} on the cone are now given in terms of the 
Laguerre polynomials, which we denote by $\sL_{\kb,n}$ and they are
\begin{equation*} 
 \sL_{\kb,n} (x,t) = L_{n-m}^{2m + 2\mu + \b+d-1}(t) t^m P_{\kb}^m\left(\frac{x}{t}\right), \quad |\kb| = m, \,\, 0 \le m \le n. 
\end{equation*}
 
\subsubsection{Jacobi polynomials on the cone}
In this case, the cone is bounded, 
$$
 \VV^{d+1} = \left \{(x,t): \|x\| \le t, \,\, x \in \BB^d, \,\, 0 \le t \le 1 \right\} ,
$$
and the weight function $w(t) = t^\b (1-t)^\g$ with $\b > - d$ and $\g > -1$. Now $W_\mu$ in~\eqref{eq:Wmu-cone} becomes 
\begin{equation*}
 W_{\b,\g,\mu}(x,t) = (t^2-\|x\|^2)^{\mu-\f12} t^\b (1-t)^\g, \quad \mu > -\tfrac12, \,\, \g > -1, \,\, \b > -d , 
\end{equation*}
and its normalization constant becomes $b_{\b,\g,\mu}^J$ given by 
\begin{equation}\label{eq:bJ}
 b_{\b,\g,\mu}^J = c_{2\mu + \b+ d-1, \g} b_\mu^\BB, \quad \hbox{with} \quad 
 c_{\a,\g} = \frac{\Gamma(\a+\g+2)}{\Gamma(\a+1)\Gamma(\g+1)}.
\end{equation}

Recall that the Jacobi polynomial $P_n^{(\a,\b)}$ is given by, for $\a, \b > -1$, 
$$
 P_n^{\a,\b}(t) = \frac{(\a+1)_n}{n!} {}_2F_1\left (\begin{matrix} -n, n+\a +\b+1 \\ \a+1\end{matrix}; \frac{1-t}{2}\right) ,
$$
in terms of the hypergeometric function and it satisfies the orthogonal relation 
$$
 \frac{c_{\a,\b}}{2^{\a+\b+1}} \int_{-1}^1 P_{n}^{(\a,\b)}(t) P_m^{(\a,\b)}(t)(1-t)^\a (1+t)^\b \d t = h_n^{(\a,\b)} \delta_{n,m},
$$
where the norm square is given by 
$$
 h_n^{(\a,\b)} = \frac{(\a+1)_n(\b+1)_n (\a+\b+n+1)}{n!(\a+\b+2)_n (\a+\b+2n+1)}.
$$ 
The orthogonal polynomials $\sQ_{\kb,n}$ in \eqref{eq:sQcone} on the cone are given in terms of the Jacobi
polynomials, which we denote by $\sJ_{\kb,n}$ and they are 
\begin{equation*}
 \sJ_{\kb,n}(x,t) = P_{n-m}^{(2m + 2\mu + \b+d-1, \g)}(1-2t) t^m P_{\kb}^m\left(\frac{x}{t}\right), \quad |\kb| = m, \,\, 0 \le m \le n. 
\end{equation*} 

\section{Monomial orthogonal polynomials on the cone}
\setcounter{equation}{0}

This section aims to construct monomial orthogonal polynomials on the cone. To illustrate our construction, 
we consider the case $d =1$ in the first subsection and present our construction for $d > 1$ in the second 
subsection for the setting of $W_\mu$ with a generic weight $w$. We specialize the result
to the Laguerre and the Jacobi polynomials on the cone in the third subsection. 

\subsection{Monomial orthogonal polynomials for $d=1$} 

Our construction is based on the orthogonal polynomials $\sS_{\kb,n}$ defined in \eqref{eq:S-OPcone}. To
illustrate our approach, we consider the case $d =1$ first, for which 
$$
 W_\mu(x,t) = w(t) (t^2-x^2)^{\mu-\f12}, \quad |x| \le t, \quad \mu > -\tfrac12,
$$
and we assume the polynomial $\sS_{\kb,n}$ takes the form 
\begin{align*}
 \sS_{k,n}(x,t) := q_{n-k}^{2k+2\mu}(t) t^k C_k^{\mu}\left(\f{x}{t} \right), \quad 0 \le k \le n, 
\end{align*} 
where $q_{n-k}^{2k+2\mu}(t) = t^{n-k} + \cdots$ is a monic polynomial of degree $n-k$ and it is orthogonal
with respect to $t^{2k+2\mu}w(t)$ on $\RR$, and we retain $C_k^\mu$ instead of using the monic Gegenbauer 
polynomial for convenience. 

\subsubsection{The case of a generic weight $w$} 
Here we construct monomial orthogonal polynomials for $W_\mu$ with a generic function $w$. 
The polynomial $\sS_{0,n}(x,t) = q_n^{2\mu}(t)= t^n + \cdots$ is a monomial orthogonal polynomial of degree $n$
in the $(x,t)$ variable. For $k > 0$, $\sS_{k,n}$ is not monomial since the polynomial $t^k C_k^{\mu}\left(\f{x}{t}\right)$ 
is a homogeneous polynomials of degree $k$ in the $(x,t)$ variables. For 
$k > 0$, we use the expansion of the Gegenbauer polynomials given by \cite[(18.5.10)]{DLMF}
\begin{align*}
C_{n}^{\mu} (x) = \sum_{i=0}^{\lfloor \f{n}{2} \rfloor} b_{n,i}^\mu x^{n - 2 i} \quad \hbox{with} \quad 
 b_{n,i}^\mu = (-1)^i 2^{n - 2 i} \f{(\mu)_{n-i}}{i! (n-2i)!},
\end{align*}
which can also be reversed to give \cite[(18.18.17)]{DLMF}
\begin{align*}
 x^k = \sum_{j=0}^{\lfloor \f{k}{2} \rfloor } c_{k,j}^\mu C_{k-2j}^{\mu} (x) \quad \hbox{with} \quad 
 c_{k,j}^\mu = \frac{k!}{2^k} \f{k-2j+\mu}{(\mu)_{k-j+1} j!}. 
\end{align*}
The above formulas hold for $\mu \ne 0$ and $\mu > -\f12$. If $\mu = 0$, then $C_n^\mu$ becomes the
Chebyshev polynomial $T_n$, for which the above relations become
\begin{align*}
 x^k & = \sum_{j=0}^{\lfloor \f{k}{2} \rfloor } c_{k,j}^0 T_{k-2j} (x) \quad \hbox{with} \quad 
 c_{k,j}^0 = \frac{1}{2^k} \binom{k}{j} \begin{cases} 2 & k \ne 2j, \\
 1 & k = 2j , \end{cases} 
 \\
T_{n} (x) & = \sum_{i=0}^{\lfloor \f{n}{2} \rfloor} b_{n,i}^0 x^{n - 2 i} \quad \hbox{with} \quad 
 b_{n,i}^0 = (-1)^i 2^{n - 2 i-1} n \f{(n-i-1)!}{i! (n-2i)!}, \quad n \in \mathbb N .
\end{align*}

\begin{thm}
For $0\le k \le n$, define 
\begin{align*}
 \sV_{k,n}(x,t) = \sum_{j=0}^{\lfloor \frac{k}{2} \rfloor } c_{k,j}^\mu \sS_{k-2j,n}(x,t).
\end{align*}
Then $\sV_{k,n}(x,t) = t^{n-k}x^k + \cdots$ is a monomial polynomial of degree $n$ and 
$\{\sV_{k,n}: 0 \le k \le n\}$ is the monomial basis of $\CV_n(\VV^2, W_\mu)$. 
\end{thm}

\begin{proof}
Since $q_{n-k+2j}^{2(k-2j) + 2\mu}(t) - t^{n-k+2j} = R_{n-k+2j-1}(t)$, which is a polynomial of degree 
$n-k+2j-1$, it follows from the definition of $c_{k,j}$ that 
\begin{align*}
 \sV_{k,n}(x,t) \, & = \sum_{j=0}^{\lfloor \frac{k}{2} \rfloor} c_{k,j}^\mu q_{n-k+2j}^{2(k-2j) + 2\mu}(t) t^{k-2j}
 C_{k-2j}^{\mu}\left(\f{x}{t} \right)\\
 & = t^n \sum_{j=0}^{\lfloor \frac{k}{2} \rfloor} c_{k,j}^\mu 
 C_{k-2j}^{\mu}\left(\f{x}{t} \right) + \sum_{j=0}^{\lfloor \frac{k}{2} \rfloor} c_{k,j}^\mu R_{n-k+2j-1}(t) 
 t^{k-2j} C_{k-2j}^{\mu}\left(\f{x}{t} \right) \\
 & = t^{n-k} x^{k} + \sum_{j=0}^{\lfloor \frac{k}{2} \rfloor} c_{k,j}^\mu R_{n-k+2j-1}(t) 
 t^{k-2j}C_{k-2j}^{\mu}\left(\f{x}{t} \right).
\end{align*}
Since the second term is evidently of degree at most $n-1$, $\sV_{k,n}$ is monomial, and the set 
$\{\sV_{k,n}: 0 \le k \le n\}$ is clearly independent. Moreover, as a linear combination of $\sS_{j,n}$, 
$\sV_{k,n}$ is an element of $\CV_n(\VV^2, W_\mu)$. This completes the proof. 
\end{proof}

\begin{prop} \label{prop:error_d=1}
For $0\leq k\leq n$, the norm of $\sV_{k,n}$ is given by 
\begin{align*}
 b_\mu \int_{\VV^2} \left | \sV_{k,n}(x,t)\right |^{2}W_{\mu }(x,t) \d x \d t 
 = \sum_{i=0}^{\lfloor \frac{k}{2}\rfloor} [c_{k,i}^\mu]^2 h_{k-2i}^\mu \left \|q_{n-k+2i}^{2(k-2i)+2\mu} \right \|^2,
\end{align*}
where $h_{n-2i}^\mu$ is given by \eqref{eq:Gegen} and it can be rewritten as 
$$
h_{k-2i}^\mu=\frac{(2\mu)_k \mu}{k!(k-2i+\mu)} \frac{(-\f{k}2)_i (\frac{-k+1}{2})_i} {(\frac{-k+1}2-\mu)_i 
(\frac{-k+2}2-\mu)_i}, 
$$
and $\|q_{n-j}^{2j+2\mu}\|$ is defined by 
$$
\left \|q_{n-j}^{2j+2\mu} \right \|^2 = b_\mu^w \int_{0}^\infty \left|q_{n-j}^{2j+2\mu}(t) \right|^2
 t^{2j+2\mu} w(t) \d t.
$$
\end{prop}

\begin{proof}
For $d=1$, the polynomials $\sS_{k,n}$ are mutually orthogonal and, more precisely, 
\begin{align*}
& b_\mu \int_{\VV^2} \sS_{k,n}(x,t)\sS_{j,n}(x,t) W_\mu(x,t) \d x \d t \\
& \phantom{ola} =\, b_\mu^w \int_0^\infty q_{n-k}^{2k+2\mu}(t)q_{n-j}^{2k+2\mu}(t)w(t)t^{k+j+2\mu} \d t 
 \, c_\mu \int_{-1}^{1} C_k^\mu (x) C_j^\mu (x) (1-x^2)^{\mu -\frac{1}{2}} \d x \\
& \phantom{ola} = \, \left \| q_{n-k}^{2k+2\mu} \right \|^{2} h_k^\mu \delta_{k,j},
\end{align*}
where $h_k^\mu$ is given in \eqref{eq:Gegen}. Hence, the stated result follows from 
\begin{align*}
 b_\mu \int_{\VV^2}\left |\sV_{k,n}(x,t)\right |^{2}W_\mu(x,t)\d x \d t 
 \, & = b_\mu \int_{\VV^2}\bigg |\sum_{i=0}^{\lfloor \frac{k}{2}\rfloor }c_{k,i}^\mu S_{k-2i,n} (x,t)\bigg |^{2}W_\mu(x,t)\d x \d t \\ 
 \,& = \sum_{i=0}^{\lfloor \frac{k}{2}\rfloor}[c_{k,i}^\mu]^2 b_\mu \int_{\VV^2} \left |S_{k-2i,n}(x,t)\right |^{2}W_\mu(x,t)\d x \d t.
\end{align*}
Finally, we use $(a)_{k-2i} = {(a)_k}/{(1-a-k)_{2i}}$ and $(x)_{2j} = 2^{2j} (\frac x 2)_j (\frac{x+1}{2})_j$ to rewrite~$h_{k-2i}^\mu$ into the given form. 
\end{proof}

\subsubsection{The Laguerre polynomials on the cone}
In this case, the weight function is $w(t) = t^\b e^{-t}$ and the polynomial $q_{n-k}^{2k+2\mu}$ is given
by the monic Laguerre polynomial, denoted by $\wh L_{n-k}^\a$; more precisely, 
$$
 q_{n-k}^{2k+2\mu}(t) = \wh L_{n-k}^{2k+2\mu+\b}(t) = (-1)^{n-k} (n-k)! L_{n-k}^{2k+2\mu+\b}(t) ,
$$
so that the monomial polynomial $\sV_{k,n}^\sL$ is given by 
\begin{equation*}
 \sV_{k,n}^\sL(x,t)= \sum_{j=0}^{\lfloor \frac{k}{2} \rfloor } c_{k,j}^\mu \sS_{k-2j,n}(x,t) 
 = \sum_{j=0}^{\lfloor \frac{k}{2} \rfloor } c_{k,j}^\mu \wh L_{n-k+2j}^{2(k-2j)+2\mu+\b}(t)t^{k}C_{k}^{\mu}\left( \frac{x}{t}\right). 
\end{equation*}
It is now easy to verify that 
\begin{align*}
 \left \|q_{n-k}^{2k+2\mu} \right \|^2 \, & = \frac{1}{\Gamma(2\mu+\b+1)} 
 \int_0^\infty \left| \wh L_{n-k}^{2k+2\mu+\b}(t) \right|^2 t^{2k+2\mu+\b} \e^{-t} \d t \\
 & = (n-k)!(2\mu+\b+1)_{n+k}. 
\end{align*}
Hence, the norm of the monomial polynomial given in Proposition \ref{prop:error_d=1} can be specified. Thus, 
we obtain
\begin{align}\label{eq:Lnorm_d=2}
 b_{\b,\mu}^L & \int_{\VV^2} \left | \sV_{k,n}^\sL(x,t)\right |^{2}W_{\b,\mu }(x,t) \d x \d t \\
 & = \sum_{i=0}^{\lfloor \frac{k}{2}\rfloor} [c_{k,i}^\mu]^2 h_{k-2i}^{\mu} 
 (n-k+2i)!(2\mu+\b+1)_{n+k-2i}. \notag
\end{align}
The sum can be written as a sum of two hypergeometric functions ${}_6F_5$ evaluated at~$1$, which however 
do not have a closed-form formula. It remains to be seen if the error can be expressed by a formula similar to 
that of \eqref{eq:leastBall} on the cone. By \eqref{eq:bestapp}, the right-hand side of the identity gives the explicit
formula for the error of the least square approximation to the monomial $t^{n-|\kb|} x^\kb$ by lower degree
polynomials. Let 
$$
 E_n(f)_{\b,\mu} = \inf_{P\in \Pi_{n-1}^{d+1}} \| f - \sP \|_{L^2(\VV^{2}, W_{\b,\mu})},
$$
where the norm is defined by normalized weight function. Using \eqref{eq:Lnorm_d=2}, we can compute this
quantity when $f$ is a monomial. Below we list the first few for small $k$:
\begin{align*}
 \left [ E_n(t^n)_{\b,\mu} \right]^2 \,& = n!(\b+2\mu+1)_n; \\
 \left [ E_n(t^{n-1}x)_{\b,\mu} \right]^2 \,& = \frac{(n-1)!(\b+2\mu+1)_{n+1}}{2(\mu+1)}; \\
 \left [ E_n(t^{n-2}x^2)_{\b,\mu} \right]^2 \,& = \frac{n!(\b+2\mu+1)_n}{4(\mu+1)^2}
 + \frac{(2\mu+1)(n-2)!(\b+2\mu+1)_{n+2}}{4(\mu+1)^2(\mu+2)}; \\
 \left [ E_n(t^{n-3}x^3)_{\b,\mu} \right]^2 \,& = \frac{9(n-1)!(\b+2\mu+1)_{n+1}}{8(\mu+1)(\mu+2)^2}+
 \frac{3(2\mu+1)(n-3)!(\b+2\mu+1)_{n+3}}{8(\mu+1)(\mu+2)^2(\mu+3)}.
\end{align*}

It should be noted that, for $d =1$, a change of variables $(t,x) \mapsto (u,v)$ given by 
\begin{equation}\label{eq:map_to_R2}
 u = \frac{t+x}{2} \quad \hbox{and}\quad v = \frac{t-x}{2}
\end{equation}
maps $\VV^2$ to $\RR_+^2$ and the weight function $t^\b e^{-t}$ to 
$$
 W_\b^L = (u+v)^\b \e^{-u-v}, \qquad (u,v) \in \RR_+^2.
$$ 
In particular, when $\b = 0$, the orthogonal polynomials become the product of classical Laguerre polynomials. 
For $\b > 0$, orthogonal polynomials with respect to $W_\b^L$ are studied in \cite{IZ}. We notice, however, that
the monomial polynomial $\sV_{k,n}$ is no longer monomial under the mapping since 
$$
 \sV_{k,n} (x,t) = \sV_{k,n}(u-v,u+v) = (u+v)^{n-k} (u-v)^k + \cdots, 
$$
so that our result remains new even when $\b = 0$. 

\subsubsection{The Jacobi polynomials on the cone}
Here the weight function is $w(t)=t^\b(1-t)^\g$ and the polynomial $q_{n-k}^{2k+2\mu }$ is 
given in terms of the monic Jacobi polynomial, denoted by $\widehat{P}_{n-k}^{(\a,\b)}$; more
precisely, 
\begin{align*}
q_{n-k}^{2k+2\mu }(t) & = \f{1}{(-2)^{n-k}} \widehat{P}_{n-k}^{(2k+2\mu +\b,\g)}(1-2t)
 \\
 & = 
\frac{(-1)^{n-k}(n-k)!}{(n+k+2\mu +\b+\g+1) _{n-k}}
P_{n-k}^{(2k+2\mu+\b,\g)}(1-2t),
\end{align*}
so that the monomial polynomial ${\mathsf{V}}_{k,n}^{\mathsf{J}}$ is given by 
\begin{align*}
\sV_{k,n}^\sJ (x,t) & =\sum_{j=0}^{\lfloor \frac{k}{2}\rfloor}c_{k,j}^{\mu} \sS_{k-2j,n}(x,t) \\
 & = \sum_{j=0}^{\lfloor \frac{k}{2} \rfloor} \f{c_{k,j}^\mu}{(-2)^{n-k+2j}}
 \widehat{P}_{n-k+2j}^{(2(k-2j)+2\mu +\b,\g)}(1-2t)t^{k}C_{k}^{\mu }\left( \frac{x}{t}\right).
\end{align*}
In this case, it is easy to verify that 
\begin{align*}
\left \| q_{n-k}^{2k+2\mu }\right \|^{2}\,& =
\f{c_{2\mu+\b,\g}}{(-2)^{2n-2k}} \int_{0}^{1} \left | \wh P_{n-k}^{(2k +2\mu +\b ,\g)}(1-2t) \right |^2
 t^{2k+2\mu +\b} (1-t)^\g \d t \\
& =\frac{(n-k)! (2\mu +\beta +1)_{n+k}(\g+1)_{n-k}}{(2\mu +\b+\g+2)_{2n} (n+k+2\mu +\b +\g+1)_{n-k}}.
\end{align*}
Hence, the norm of the monomial polynomial is given by the sum 
\begin{align} \label{eq:normJ}
& b_{\beta,\gamma,\mu}^{J} \int_{\VV^2} \left | \sV_{k,n}^\sJ (x,t)\right |^2 W_{\b,\g,\mu}(x,t) \d x \d t \\
 &\quad = \sum_{i=0}^{\lfloor \frac{k}{2}\rfloor }[c_{k,i}^\mu]^2 h_{k-2i}^\mu 
 \frac{(n-k+2i)! (2\mu + \b +1)_{n+k-2i}(\g +1)_{n-k+2i}}{(2\mu + \b + \g +2)_{2n}
 (n+k-2i+2\mu + \b + \g +1) _{n-k+2i}}. \notag
\end{align}
The sum can be written as a sum of three hypergeometric functions ${}_8F_7$ evaluated at~$1$ but no
closed-form formula. It remains to be seen if it can be expressed by a formula similar to that of \eqref{eq:leastBall}. 
Let 
$$
 E_n(f)_{\b,\g,\mu} = \inf_{P\in \Pi_{n-1}^{d+1}} \| f - \sP \|_{L^2(\VV^{2}, W_{\b,\g,\mu})},
$$
where the norm is defined by normalized weight function. We can use \eqref{eq:normJ} to compute this quantity
for $f$ is a monomial. Below are the first few cases for $f(x,y) = t^{n-|\kb|}x^\kb$, in which $\a = \b+\g + 2\mu+1$,
\begin{align*}
 \left [ E_n(t^n)_{\b,\g,\mu} \right]^2 \,& = \frac{n!(\g+1)_n (\a)_n (\b+2\mu+1)_n}{n!(\a)_{2n}(\a+1)_{2n}}; \\
 \left [ E_n(t^{n-1}x)_{\b,\g,\mu} \right]^2 \,& = \frac{(n-1)!(\g+1)_{n-1}(\a+1)_n(\b+2\mu+1)_{n+1}}
 {2(\mu+1) (\a+1)_{2n} (\a+1)_{2n-1}}; \\
 \left [ E_n(t^{n-2}x^2)_{\b,\g,\mu} \right]^2 \,& = \frac{n!(\g+1)_n(\a)_n (\a+1)_n(\b+2\mu+1)_n}{4(\mu+1)^2
 (\a)_{2n} (\a+1)_{2n}} \\
 &   + \frac{(2\mu+1)(n-2)!(\g+1)_{n-2}(\a+1)_n (\a+2)_n (\b+2\mu+1)_{n+2}}
     {4(\mu+1)^2(\mu+2)(\a+1)_{2n} (\a+2)_{2n-2}}.
\end{align*}
In particular, $E_n(t^n)_{\b,\g,\mu}$ is the $L^2$ norm of 
$\|\wh P_n^{(2\mu+\b,\g)}\|_{L^2([-1,1],w_{2\mu+\b,\g})}$, as can be easily verified. 

We note that, for $d =1$, the change of variables $(t,x) \mapsto (u,v)$ in \eqref{eq:map_to_R2} sends 
$$
 (x,t) \in \VV^2 \mapsto (u,v) \in \TT^2 =\{(u,v): u\ge 0, v\ge 0, u+v \le 1\},
$$
where $\TT^2$ is the standard triangle; and sends the weight function $t^\b (1-t)^\g$ to 
$$
 W_\b^J = (u+v)^\b (1-u-v)^\g, \qquad (u,v) \in \TT^2.
$$ 
In particular, when $\b = 0$, the orthogonal polynomials on the cone become a special case of the classical
Jacobi polynomials on the triangle. For $\b \ne -1$, the orthogonal polynomials are not the classical Jacobi 
polynomials but are special cases studied in~\cite{OTV}, which has also been extended to simplexes in \cite{AAG}. 
Nevertheless, as in the Laguerre case, the monomial polynomial $\sV_{k,n}$ is no longer monomial under
the mapping, and our result remains new even when $\b = 0$. 

\subsection{Monomial orthogonal polynomials for $d>1$}
Following the idea for $d =1$, we now construct monomial orthogonal polynomials on the cone for $d > 1$. 
Again we deal with the weight function $W_\mu$ with a generic weight function $w$ first. 

\subsubsection{The case of a generic weight function $w$} We consider $W_\mu$ defined in \eqref{eq:Wmu-cone}
and use $\sS_{\kb,n}$ defined in \eqref{eq:S-OPcone} to construct monomial orthogonal polynomial $\sV_{\kb,n}$ 
that satisfies \eqref{eq:Vcone-def}. 

\begin{thm}
Let $d \ge 2$ and let $c_{\kb,\jb}^\mu$ be the coefficients defined in \eqref{eq:monic-ckj}. For $\kb \in \NN_0^d$
and $|\kb|\le n$, define 
\begin{align} \label{eq:sVk-def}
 \sV_{\kb,n} (x,t) = \sum_{2 \jb \leq \kb} c_{\kb,\jb}^\mu \sS_{\kb-2 \jb,n}(x,t).
\end{align}
Then $\sV_{\kb,n}(x,t) = t^{n-|\kb|} x^{\kb}+ \cdots$ is a monomial polynomial of degree $n$ and the
set $\{\sV_{\kb,n}: 0 \le |\kb| \le n, \, \kb \in \NN_0^d\}$ is the monomial basis of $\CV_n(\VV^{d+1}, W_\mu)$. 
Furthermore, $\sV_{\kb,n}$ satisfies
\begin{align} \label{eq:sVk-def2}
 \sV_{\kb,n} (x,t) = \sum_{2 \ib \leq \kb} \, x^{\kb- 2\ib} \sum_{m=0}^{|\ib|} 
 q_{n-|\kb|+2m}^{2(|\kb|-2 m)+ 2\mu}(t) t^{2|\ib|-2m} B_{\ib,m}, 
\end{align} 
where $B_{\ib, m}$ depends only on $|\ib|$ and is defined by, with $\l = \mu+ \f{d-1}{2}$,
\begin{align}\label{eq:Bim}
B_{\ib,m} = \frac{(-1)^m (\frac{-\kb}{2})_{\ib}(\frac{-\kb+\one}{2})_{\ib}}{(-|\kb| - \lambda+1)_{|\ib|+ m}\ib!} \binom{|\ib|}{m} 
 \left(2 (-|\kb| -\lambda+1)_{m} - (-|\kb| - \lambda)_{m} \right). 
\end{align}
\end{thm} 

\begin{proof}
Using that $q_{n-m}^{\a_m}(t) - t^{n-m} = R_{n-m-1}(t)$ is a polynomial of degree $n-m-1$, 
$\a_m = 2m + 2\mu$, we obtain, by \eqref{eq:monic-V},
\begin{align*}
 \sV_{\kb,n}(x,t) \, & = \sum_{2 \jb \leq \kb} c_{\kb,\jb}^\mu q_{n-|\kb|+2|\jb|}^{2(|\kb|-2|\jb|)+ 2\mu} (t) t^{|\kb|-2|\jb|}
 V_{\kb-2\jb}\left(\f{x}{t} \right)\\
 & = t^n\sum_{2 \jb \leq \kb} c_{\kb,\jb}^\mu V_{\kb-2\jb}\left(\f{x}{t} \right)
 + \sum_{2 \jb \leq \kb} c_{\kb,\jb}^\mu R_{n-|\kb|+2|\jb|-1}(t) 
 t^{|\kb|-2|\jb|} V_{\kb-2\jb}\left(\f{x}{t} \right) \\
 & = t^{n-|\kb|} x^{\kb} + \sum_{2 \jb \leq \kb} c_{\kb,\jb}^\mu R_{n-|\kb|+2|\jb|-1}(t) 
 t^{|\kb|-2|\jb|} V_{\kb-2\jb}\left(\f{x}{t} \right),
\end{align*}
which shows that $\sV_{\kb,n}$ is a monomial polynomial of degree $n$ and it is an element of 
$\CV_n(\VV^{d+1}, W_{\mu})$ as a linear combination of $\sS_{\kb,n}$. As a consequence, 
$\{V_{\kb,n}: 0 \le |\kb| \le n\}$ is the monomial basis for the space of orthogonal polynomials. 

We now derive the more explicit formula for this polynomial. Using the monomial expansion of $V_\kb$ in 
\eqref{eq:V-monic}, we obtain
\begin{align*}
 \sV_{\kb,n}(x,t) \, & = \sum_{2 \jb \leq \kb} c_{\kb,\jb}^\mu q_{n-|\kb|+2|\jb|}^{2(|\kb|-2|\jb|)+ 2\mu} (t)
 t^{|\kb|-2|\jb|} \sum_{2 \ib \leq \kb-2 \jb} b_{\kb-2\jb,\ib}^\mu \, \left( \frac{x}{t}\right)^{\kb- 2\jb - 2 \ib} \\
 & = \sum_{2 \jb \leq \kb} c_{\kb,\jb}^\mu q_{n-|\kb|+2|\jb|}^{2(|\kb|-2|\jb|)+ 2\mu} (t)
 \sum_{2 \ib \leq \kb} t^{2|\ib|-2|\jb|} b_{\kb-2\jb,\ib-\jb}^\mu \, x^{\kb- 2\ib} \\
 & = \sum_{2 \ib \leq \kb} \, x^{\kb- 2\ib} \sum_{\jb \leq \ib} q_{n-|\kb|+2|\jb|}^{2(|\kb|-2|\jb|)+ 2\mu} (t)
 t^{2(|\ib|-|\jb|)} c_{\kb,\jb}^\mu b_{\kb-2\jb,\ib-\jb}^\mu \\
 & = \sum_{2 \ib \leq \kb} \, x^{\kb- 2\ib} \sum_{m=0}^{|\ib|} 
 q_{n-|\kb|+2|\jb|}^{2(|\kb|-2|\jb|)+ 2\mu} (t) t^{2|\ib|-2m} 
 \sum_{|\jb|=m} c_{\kb,\jb}^\mu b_{\kb-2\jb,\ib-\jb}^\mu.
\end{align*}
We now evaluate the last sum in the statement, which is the $B_{\ib,m}$. First, we rewrite $b_{\kb-2\jb,\ib-\jb}^\mu$. 
By \eqref{eq:V-monic} and using 
$$
 (a+m)_{n-m} = \frac{(a)_n}{(a)_m} \quad \hbox{and} \quad (a+ 2m)_{n-m} = \frac{(a)_{n+m}}{(a)_{2m}},
$$
as well as 
$$
 (-|\kb-2 \jb| - \lambda+1)_{|\ib-\jb|} = (-|\kb|+2 |\jb| - \lambda+1)_{|\ib|-|\jb|} 
 = \frac{(-|\kb|- \lambda+1)_{|\ib|+|\jb}}{(-|\kb - \lambda+1)_{2|\jb|}},
$$
we obtain
\begin{align*}
 b_{\kb-2\jb, \ib-\jb}^\mu \, & = \frac{(\frac{-\kb+2\jb)}{2})_{\ib -\jb} 
 (\frac{-\kb+2\jb+1 )}{2})_{\ib -\jb}}{(-|\kb-2 \jb| - \lambda+1)_{|\ib-\jb|} (\ib-\jb)!} \\
 & = \frac{(\frac{-\kb}{2})_{\ib}}{(\frac{-\kb}{2})_{\jb}} \frac{(\frac{-\kb +1}{2})_{\ib}}{(\frac{-\kb+1}{2})_{\jb}} 
 \frac{(-|\kb| - \lambda+1)_{2|\jb|}}{(-|\kb| - \lambda+1)_{|\ib|+|\jb|}(\ib-\jb)!}.
\end{align*}
Hence, using \eqref{eq:monic-V}, it follows that 
\begin{align*}
B_{\ib,m} & = \sum_{|\jb|=m} c_{\kb,\jb}^\mu b_{\kb-2 \jb, \ib-\jb}^\mu \,
 \\
 & = \left(2 (-|\kb| -\lambda+1)_{m} - (-|\kb| - \lambda)_{m} \right)
\frac{(-1)^m (\frac{-\kb}{2})_{\ib}(\frac{-\kb+\one}{2})_{\ib}}{(-|\kb| - \lambda+1)_{|\ib|+ m}} 
 \sum_{|\jb|=m} \frac{1}{(\ib-\jb)! \jb!}.
\end{align*}
By the binomial formula of $d$ variables, the last sum is equal to $\frac{1}{\ib!} \binom{|\ib|}{m}$. 
Putting these together proves the theorem. 
\end{proof}

The expression of $\sV_{\kb,n}$ in \eqref{eq:sVk-def2} is closest we get for the monomial expansion of
$\sV_{\kb,n}$. Using the explicit formula, we can compute the norm of the polynomial $\sV_{\kb,n}$.

\begin{prop} \label{prop:error_d>1}
For $\kb \in \NN_0^d$, $k = |\kb|$, and $0\leq k\leq n$, the norm of $\sV_{\kb,n}$ is given by 
\begin{align*}
 b_\mu \int_{\VV^{d+1}} & \left | \sV_{\kb,n}(x,t)\right |^{2}W_{\mu}(x,t) \d x \d t \\
 &
 = \frac{(\frac{\one}2)_{\kb}}{(\mu+\frac{d+1}2)_k} 
 \sum_{2 \ib \leq \kb} \sum_{m=0}^{|\ib|} B_{\ib,m} \|q_{n-k+2m}^{2(k-2 m)+ 2\mu}\|^2
 \frac{(-k-\mu-\frac{d-1}2)_{|\ib|}}{(-\kb+\f12)_\ib},
\end{align*}
where $\|q_{n-j}^{2j+2\mu}\|$ is defined by 
$$
\left \|q_{n-j}^{2j+2\mu} \right \|^2 = b_\mu^w \int_{0}^\infty \left|q_{n-j}^{2j+2\mu}(t) \right|^2
 t^{2j+2\mu} w(t) \d t.
$$
\end{prop}

\begin{proof}
Since $\sV_{\kb,n}$ is monomial orthogonal polynomial, we obtain 
\begin{align*} 
 \mathrm{LH}: = b_\mu \int_{\VV^{d+1}} & \left[\sV_{\kb,n}(x,t)\right]^2 W_\mu(x,t) \d x \d t \\
 & = b_\mu \int_{\VV^{d+1}} \sV_{\kb,n}(x,t) t^{n-k} x^\kb W_\mu(x,t) \d x \d t. 
\end{align*}
By \eqref{eq:sVk-def2} and setting $x = t y$, it follows that 
\begin{align*}
 \mathrm{LH}:= & \sum_{2 \ib \leq \kb} b_\mu \int_{\VV^{d+1}} x^{\kb- 2\ib} \sum_{m=0}^{|\ib|} 
 q_{n-k+2m}^{2(k-2 m)+ 2\mu}(t) t^{2|\ib|-2m} B_{\ib,m} t^{n-k}x^\kb W_\mu(x,t) \d x \d t \\
 = & \sum_{2 \ib \leq \kb} \sum_{m=0}^{|\ib|} B_{\ib,m} b_\mu^w \int_0^\infty 
 q_{n-k+2m}^{2(k-2 m)+ 2\mu}(t) t^{n+k-2m+2\mu+d-1} w(t) \d t \\
 &\qquad \qquad\quad \times b_\mu^\BB \int_{\BB^d} y^{2\kb - 2\ib} (1-\|y\|^2)^{\mu-\f12} \d y \\
 = & \sum_{2 \ib \leq \kb} \sum_{m=0}^{|\ib|} B_{\ib,m} \|q_{n-k+2m}^{2(k-2 m)+ 2\mu}\|^2
 b_\mu^\BB \int_{\BB^d} y^{2\kb - 2\ib} (1-\|y\|^2)^{\mu-\f12} \d y. 
\end{align*}
For the last integral, let $\TT^d = \{y \in \RR^d: y_i \ge 0, 1\le i \le d, \, 1- |y| \ge 0\}$,
where $|y| = y_1+\ldots + y_d$, be the simplex in $\RR^d$. Using Lemma 4.4.1 in \cite{DX}, we obtain 
that 
\begin{align*}
 b_\mu^\BB \int_{\BB^d} y^{2\kb - 2\ib} (1-\|y\|^2)^{\mu-\f12} \d y 
 \, & = b_\mu^\BB \int_{\TT^d} y^{\kb - \ib -\f12} (1-|y|)^{\mu-\f12} \d y \\
 \,& = \frac{(\frac{\one}2)_{\kb-\ib}}{(\mu+\frac{d+1}2)_{k-|\ib|}} 
 = \frac{(\frac{\one}2)_{\kb}(-\mu-\frac{d-1}2 -k)_{|\ib|}}{(-\kb+\f12)_\ib (\mu+\frac{d+1}2)_{k}}.
\end{align*}
Putting together, this completes the proof.
\end{proof}
Our next proposition shows that \eqref{eq:sVk-def} can also be reversed. 

\begin{prop} \label{prop:sS-sV}
Let $b_{\kb,\jb}$ be the coefficients in \eqref{eq:V-monic}. Then, for $\kb \in \NN_0^d$ and $|\kb| \le n$, 
$$
 \sS_{\kb,n}(x,t) = \sum_{2 \jb \leq \kb} b_{\kb, \jb}^\mu \sV_{\kb - 2 \jb,n} (x,t). 
$$
\end{prop}

This follows from a simple computation: 
\begin{align*}
\sum_{2 \jb \leq \kb} b_{\kb, \jb}^\mu \sV_{\kb - 2 \jb,n} (x,t) & =
 \sum_{2 \jb \leq \kb} b_{\kb, \jb}^\mu \sum_{2 \ib \leq \kb - 2 \jb} c_{\kb- 2 \jb, \ib}^\mu \sS_{\kb -2 \jb - 2 \ib,n}(x,t) \\
 & = \sum_{2 \jb \leq \kb} b_{\kb, \jb}^\mu \sum_{2 \ib \leq \kb} c_{\kb - 2 \jb, \ib - \jb}^\mu
 \sS_{\kb - 2 \ib,n}(x,t) \\
 & = \sum_{2 \ib \leq \kb} \sS_{\kb - 2 \ib,n}(x,t)
 \sum_{ \jb \leq \ib} b_{\kb, \jb}^\mu c_{\kb - 2 \jb, \ib - \jb}^\mu \\
 & = \sS_{\kb,n}(x,t),
\end{align*}
where the last step follows from \eqref{eq:bc=delta}.

\subsubsection{Monomial Laguerre polynomials}
For the Laguerre weight function $w(t) = t^\b e^{-t}$, we use monic Laguerre polynomials $\wh L_{n-k}^\a$,
so that the monomial polynomial, denoted by $\sV_{\kb,n}^\sL$ with $\sL$ standing for Laguerre, becomes,
by \eqref{eq:sVk-def2}, 
\begin{equation*}
 \sV_{\kb,n}^\sL (x,t) = \sum_{2 \jb \leq \kb} \, x^{\kb- 2\jb} \sum_{m=0}^{|\jb|} B_{\jb,m}
 \wh L_{n-|\kb|+2m}^{2(|\kb|-2 m) + \b+ 2\mu+d-1}(t) t^{2|\jb|-2m} , 
\end{equation*}
where $B_{\jb,m}$ is given by \eqref{eq:Bim}. As in the case of $d =1$, the norm of the monic 
Laguerre polynomial is given by 
$$
 \left \|\wh L_{n-|\kb|+2m}^{2(|\kb|-2 m) + \b+ 2\mu+d-1} \right\|^2 = (n-|\kb|+2m)!(2\mu+d+\b)_{n+|\kb|-2m}.
$$
Hence, for $\kb \in \NN_0^d$, $k = |\kb|$, and $0\leq k\leq n$, the norm of $\sV_{\kb,n}^L$ given in 
Proposition \ref{prop:error_d>1} is specified to 
\begin{align*}
 b_\mu & \int_{\VV^{d+1}} \left | \sV_{\kb,n}^\sL (x,t)\right |^{2}W_{\b,\mu}(x,t) \d x \d t \\
 & = \frac{(\frac{\one}2)_{\kb}}{(\mu+\frac{d+1}2)_k} 
 \sum_{2 \ib \leq \kb} \sum_{m=0}^{|\ib|} B_{\ib,m} 
 \frac{(-k-\mu-\frac{d-1}2)_{|\ib|}}{(-\kb+\f12)_\ib} (n-k+2m)!(2\mu+d+\b)_{n+k-2m},
\end{align*}
where $\kb \in \NN_0^d$, $k = |\kb|$, and $0\leq k\leq n$, 

\subsubsection{Monomial Jacobi polynomials}
For the Jacobi weight function $w(t) = t^\b (1-t)^\g$, we use monic Jacobi polynomials 
$\wh P_{n-k}^{(\a,\b)}(1-2t)$, so that the monomial polynomial, denoted by $\sV_{k,n}^\sJ$, becomes
by \eqref{eq:sVk-def2} 
\begin{equation*}
 \sV_{\kb,n}^\sJ (x,t) = \sum_{2 \jb \leq \kb} \, x^{\kb- 2\jb} \sum_{m=0}^{|\jb|} \f{B_{\jb,m}}{(-2)^{n-| \kb | +2m}}
 \wh P_{n-|\kb|+2m}^{(2(|\kb|-2 m) +\b+ 2\mu+d-1, \g)}(1-2 t) t^{2|\jb|-2m}, 
\end{equation*}
where $B_{\jb,m}$ is given by \eqref{eq:Bim}.
As in the case of $d =1$, the norm of the monic Jacobi polynomial is given by 
\begin{align*}
\left \|q_{n-|\kb|+2m}^{2(|\kb|-2 m) + 2\mu} \right\|^2
 &=
\left \| \f{\wh P_{n-|\kb|+2m}^{(2(|\kb|-2 m) + \b + 2\mu +d-1 , \g)}}{(-2)^{n-| \kb | +2m}} \right\|^2 \\
 & = \frac{(n-|\kb|+2m)! (2\mu + \b +d-1)_{n+|\kb|-2m}(\g +1)_{n-|\kb|+2m}}{(2\mu + \b + \g +d+1)_{2n}
 (n+|\kb|-2m+2\mu + \b + \g +d) _{n-|\kb|+2m}}.
\end{align*}
Hence, for $\kb \in \NN_0^d$, $k = |\kb|$, and $0\leq k\leq n$, the norm of $\sV_{\kb,n}^\sJ$ given in 
Proposition \ref{prop:error_d>1} is specified to 
\begin{align*} 
& b_{\beta,\gamma,\mu}^{J} \int_{\VV^{d+1}} \left | \sV_{k,n}^\sJ (x,t)\right |^2 W_{\b,\g,\mu}(x,t) \d x \d t \\
 & = \sum_{i=0}^{\lfloor \frac{k}{2}\rfloor }[c_{k,i}^\mu]^2 h_{k-2i}^\mu 
 \frac{(n-k+2i)! (2\mu + \b +d-1)_{n+k-2i}(\g +1)_{n-k+2i}}{(2\mu + \b + \g +d+1)_{2n}
 (n+k-2i+2\mu + \b + \g +d) _{n-k+2i}}. \notag
\end{align*}

\section{Orthogonal polynomials via Rodrigues formulas}
\setcounter{equation}{0}

The Laguerre and Jacobi polynomials can be defined by the Rodrigues formulas, which we use to define
$U$-family of orthogonal polynomials on the cone. 

\subsection{Laguerre polynomials on the cone}
Recall that on the infinite cone $\VV^{d+1}$, the Laguerre weight function is defined by 
$$
 W_{\b,\mu}(x,t) = (t^2-\|x\|^2)^{\mu-\f 12} t^\b \e^{-t}, \qquad \mu > -\f12,\quad \b > -d .
$$
We define a family of polynomials $\sU_{\kb,n}^\sL$ by the Rodrigues type formula. 

\begin{defn}\label{def:UknL}
Let $\mu > -\f12$ and $\b > -d $. For $\kb \in \NN_0^d$ and $|\kb| = k$, $0 \le k\le n$, define
$$
\sU_{\kb,n}^\sL (x,t) =\f{1}{ W_{\b,\mu}(x,t)} \f{1}{t^{2k+2\mu+d-1}}
 \f{\partial^{\kb}}{\partial x^{\kb}} (t^2-\| x \|^2)^{k + \mu -\f{1}{2}}
 \f{\partial^{n-k}}{\partial t^{n-k}} t^{n+k+\beta+2\mu+d-1}\e^{-t}.
$$
\end{defn}

We note that the order of the derivatives in the above definition cannot be exchanged since $(t^2-\|x\|^2)^{k+\mu-\f12}$
depends on both $t$ and $x$. 

\begin{thm} \label{prop:UkL}
The family of polynomials $\{\sU_{\kb,n}^\sL: |\kb| \le n\}$ is a basis of $\CV_n(\VV^{d+1}, W_{\b,\mu})$. 
\end{thm}

\begin{proof}
The definition of $\sU_{\kb,n}^\sL$ can be rewritten as 
\begin{align*}
\sU_{\kb,n}^\sL(x,t) = \,& \f{1}{(t^2-\|x\|^2)^{\mu - \f12}}\f{\partial^{\kb}}{\partial x^{\kb}} (t^2-\| x \|^2)^{k + \mu -\f{1}{2}} \\
 & \times \f{1}{t^{2k+\b+2\mu+d-1} \e^{-t}} \f{\partial^{n-k}}{\partial t^{n-k}} ( t^{n+k+\beta+2\mu+d-1}\e^{-t}),
\end{align*}
from which it follows immediately that $\sU_{\kb,n}^\sL$ is indeed a polynomial of degree $n$. We now prove that
it is an orthogonal polynomial. 

Let $P$ be a polynomial of total degree at most $n-1$ in the $(x,t)$ variables. By the definition of $\sU_{\kb,n}^\sL$,
\begin{align*}
 I & :=\int_{\VV^{d+1}} \sU_{\kb,n}^\sL (x,t)P(x,t)W_{\beta,\mu }(x,t)\d x \d t 
 \\
 & = \int_{\VV^{d+1}} P (x,t)\frac{\partial^{\kb}}{\partial x^\kb}
(t^{2}- \|x\|^{2})^{k+\mu -\frac{1}{2}} \d x
\frac{1}{t^{2k+2\mu+d-1}}\frac{\partial^{n-k}}{\partial t^{n-k}}%
 \left(\e^{-t}t^{n+k+2\mu +\beta +d-1}\right)\d t,
\end{align*}
which becomes, after integration by parts $k$ times with respect to $x$, 
\begin{align*}
 I = (-1)^{k} \int_{\VV^{d+1}} & \frac{1}{t^{2k+2\mu +d-1}}(t^2- \|x\|^2)^{k+\mu -\frac{1}{2}} 
 \\ &\times 
\frac{\partial^{n-k}}{\partial t^{n-k}} \left(\e^{-t}t^{n+k+2\mu +\beta +d-1}\right) Q(x,t) \d x \d t,
\end{align*}
where $Q = \frac{\partial^\kb}{\partial x^\kb}P$ is a polynomial of total degree at most 
$n-k-1$ in the $(x,t)$ variables. Under the substitution $x= ty,$ it follows
\begin{align*}
 I & = (-1)^k \int_0^\infty \int_{\BB^d} (1-\|y\|^{2})^{k+\mu -\f12} \frac{\partial ^{n-k}}{\partial t^{n-k}}
 \left(\e^{-t}t^{n+k +2\mu +\b +d-1}\right) Q (ty,t) \d y \d t. 
\end{align*}
Since $Q$ is a polynomial of total degree $n-k-1$ in the $(x,t)$ variables, $Q(ty,t)$ is a polynomial of 
degree at most $n-k-1$ in the $t$ variable. In particular, $\frac{\partial^{n-k}}{\partial t^{n-k}}Q =0$.
Hence, applying integration by parts $n-k$ times with respect to $t$, we conclude that $I=0$. This 
completes the proof. 
\end{proof}

Unlike orthogonal polynomials on the unit ball, the two bases of $\{\sU^\sL_{\kb,n}\}$ and $\{\sV^\sL_{\kb,n}\}$
are not mutually biorthogonal. They satisfy the following partial biorthogonality.

\begin{prop}\label{prop:biorthL}
The polynomials $\sU^\sL_{\kb,n}(x,t)$ and $\sV^\sL_{\kb,n} (x,t)$ are partial biorthogonal in the sense
that, for $\kb, \jb \in \NN_0^d$ with $k = |\kb| \le n$ and $j= |\jb| \le n$, the relation
\begin{align*}
 b_{\b,\mu}^L \int_{\VV^{d+1}} & \sU^\sL_{\kb,n}(x,t) \sV^\sL_{\jb,n}(x,t) W_{\b,\mu}(x,t) \d x \d t \\
 = & k! (n-k)! (-1)^n (2\mu +\beta +d) _{n+k}\frac{(\mu+\frac{1}{2})_{k}}{(\mu +\frac{d+1}{2}) _k} \delta _{\kb,\jb},
\end{align*}
holds if $k \ge j$ or $j > k$ and $\jb- \kb$ has an odd component.
\end{prop}

\begin{proof}
Since $\sU^{\sL}_{\kb,n}$ is an orthogonal polynomial of degree $n$ and $\sV^\sL_{\kb,n}$ is monomial, 
we obtain immediately that 
\begin{align*}
\left \langle \sU^{\sL}_{\kb,n}, \sV^{\sL}_{\jb,n} \right \rangle 
\,& = b_{\b,\mu}^{L} \int_{\VV^{d+1}} \sU^{\sL}_{\kb,n} (x,t)\sV^{\sL}_{\jb,n}(x,t)W_{\beta,\mu}(x,t) \d x \d t \\
& = b_{\b,\mu}^{L} \int_{\VV^{d+1}} \sU^{\sL}_{\kb,n}(x,t) t^{n-|\jb|}x^\jb W_{\beta,\mu}(x,t) \d x \d t. 
\end{align*}
Integrating by parts $\kb$ times with respect to $x$ as in the previous theorem, we obtain 
\begin{align*}
\left \langle \sU^{\sL}_{\kb,n}, \sV^{\sL}_{\jb,n} \right \rangle =\,& (-1)^k b_{\b,\mu}^L
 \int_{\VV^{d+1}} \frac{t^{n-k}}{t^{2k+2\mu +d-1}}(t^{2}-\|x\|^2)^{k +\mu -\frac{1}{2}}\\
&\times \frac{\partial ^{n-k}}{\partial t^{n-k}} \left(\e^{-t} t^{n+k+2\mu +\beta +d-1}\right)
\frac{\partial^\kb} {\partial x^\kb} x^{\jb} \d x \d t,
\end{align*}
which is zero if $k > j$. If, however, $j > k$ and there is an $i$ such that $j_i - k_i$ is odd, then
$\frac{\partial^\kb} {\partial x^\kb} x^{\jb} = \frac{\jb!}{(\jb - \kb)!} x^{\jb-\kb}$ has a component $x_i^{j_i-k_i}$ of odd power, so that the last integral is zero if we make a change of variable $x_i \mapsto - x_i$.
Finally, for $\jb = \kb$, 
we use substitution $x=ty$ and integrate by parts with respect to $t$, as in the previous proof, to obtain 
\begin{align*}
\left \langle \sU^{\sL}_{\kb,n}, \sV^{\sL}_{\kb,n} \right \rangle 
 = \, & (-1)^k \kb! b_\mu^\BB \int_{\mathbb{B}^{d}} (1- \|y\|^2)^{k+\mu -\frac12}\d y \\
 & \times \, \frac{1}{\Gamma(2\mu + \b+d)} \int_0^\infty t^{n- k}
 \frac{\partial ^{n-k}}{\partial t^{n-k}} \left(\e^{-t} t^{n+k+2\mu +\beta +d-1}\right) \d t \\
 = \, & (-1)^n \kb! \frac{b_\mu^\BB}{b_{\mu+k}^\BB} 
 \frac{(n-k)!}{\Gamma(2\mu + \b+d)} \int_0^\infty \e^{-t} t^{n+k+2\mu +\beta +d-1} \d t \\
 = \, & (-1)^n \kb! (n-k)! (2\mu+\b+d)_{n+k} \frac{(\mu+\frac{1}{2}) _{k}}{(\mu +\frac{d+1}{2}) _k},
\end{align*}
which completes the proof. 
\end{proof}
 
We can find a basis in terms of $\sU_{\kb,n}^\sL$ that is biorthogonal with respect to the basis of $\sV_{\kb,n}^\sL$. This is illustrated in the following for the case $d = 1$.

\begin{prop}
For $d =1$, define the polynomials $\{\sY_{n-k,n}^\sL: 0 \le k \le n\}$ recursively as follows: for
$k= 0, 1$, $\sY_{n-k,n}^\sL = \sU_{n-k, n}^\sL$, and for $k =2, 3,\ldots,n$,
\begin{equation} \label{eq:sYkn}
 \sY_{n-k,n}^\sL (x,t) = \sU_{n-k,n}^{\sL}- \sum_{1 \le i \le  \lfloor k/2\rfloor} c_{i,k} \sY_{n-k+2i,n}^\sL 
\end{equation}
recursively, where the coefficients are 
\begin{equation*}
c_{i,k}=\frac{k!\left( \mu +i+n-k+1\right) _{i}}{2^{2i} 
 \left( k-2i\right)!i!\left( \mu +n-k+1/2\right) _{2i}\left( 2\mu +\beta +2n-k+1\right)_{2i}}.
\end{equation*}%
Then $\{\sY_{n-k,n}^\sL: 0 \le k \le n\}$ is a basis of $\CV_n(\VV^2,W_{\b,\mu})$ and it is mutually 
biorthogonal to the basis $\{\sV_{k,n}^\sL: 0 \le k \le n\}$.
\end{prop} 

\begin{proof}
We define $\sY_{n-k,n}^\sL$ as in \eqref{eq:sYkn} and determine the coefficients $c_{i,k}$ by requiring 
$$
I_{n-k,n-j} := \la \sY^{\sL}_{n-k,n}, \sV^{\sL}_{n-j,n} \ra = h_{n-k} \delta_{k,j}, 
$$
where, by \eqref{eq:sYkn} and the required biorthogonality,
$$
 h_{n-k} = \la \sU^{\sL}_{n-k,n}, \sV^{\sL}_{n-k,n} \ra. 
$$
The case $k = 0$ and $1$ holds trivially by Proposition \ref{prop:biorthL}. Assume that $\sY_{n-k+2i,n}^\sL$ 
has been determined for $i = 1, \ldots , \lfloor k/2 \rfloor$ for $2 \le k \le n$. We now define $\sY_{n-k,n}^\sL$
as given in~\eqref{eq:sYkn}. Then
$$
 I_{n-k, n-j} = \la \sU_{n-k,n}^\sL, \sV_{n-j,n}^\sL \ra - h_{n-j} \sum_{1 \le l \le \lfloor k/2 \rfloor} c_{l,k} \delta_{n-k+2l, n-j},
$$
which is zero if $j > k$ or $k$ and $j$ have different parity when $j < k$. Thus, we only need to consider 
$j = k-2i$, for which $I_{n-k,n-j} = 0$ becomes 
$$
 0 = \la \sU_{n-k,n}^\sL, \sV_{n-k+2i,n}^\sL \ra - c_{i,k} h_{n-k+2i}, \quad 1 \le i \le  \lfloor k/2 \rfloor,
$$
which gives immediately 
$$
c_{i,k} = \la \sU_{n-k,n}^\sL, \sV_{n-k+2i,n}^\sL \ra / h_{n-k+2i}.
$$ 
The value of $h_{n-k+2i}$ is already given in Proposition \ref{prop:biorthL}. Furthermore, following the 
proof there, we also obtain 
\begin{align*}
\left \langle \sU_{k,n}^\sL, \sV_{k+2i,n}^\sL \right \rangle 
 = & \, b_{\b,\mu}^L(-1)^k \frac{(k+2i)!}{(2i)!} \int_{-1}^1 y^{2i} (1-y^2)^{k+\mu-\f12} \d y \\
& \times \int_0^\infty t^{n-k} \frac{\partial^{n-k}} {\partial t^{n-k}} \left( t^{n+k+\b+2\mu}e^{-t} \right) \d t \\
= & (-1)^n \frac{(k+2 i)!(n-k)! (2\mu+\b+1)_{n+k}(\mu+\f12)_k (\f12)_i}{(2i)!(\mu+1)_{k+i}}, 
\end{align*}
where the last step follows from integrating by parts $n-k$ times in the $t$ variable and evaluating the
two integrals of one variable. Using this identity with $k$ replaced by $n-k$, we obtain the formula of
the coefficient $c_{i,k}$ after rewriting the Pochhammer symbols. 
\end{proof}

\subsection{Jacobi polynomials on the cone}
Recall that on the finite cone $\VV^{d+1}$, the Jacobi weight function is defined by 
$$
 W_{\b,\g,\mu}(x,t) = (t^2-\|x\|^2)^{\mu-\f 12} t^\b (1-t)^\g, \qquad \mu > -\f12,\quad \b > -d ,\quad \g > -1.
$$
We define a family of polynomials $\sU_{\kb,n}^\sL$ by the Rodrigues type formula. 

\begin{defn}
Let $\mu > -\f12$, $\b > -d$ and $\g > -1$. For $\kb \in \NN_0^d$ and $|\kb| = k$, $0 \le k\le n$, define
\begin{multline*}
\sU_{\kb,n}^\sJ (x,t) =\frac{1}{W_{\b,\g,\mu}(x,t)}
 \frac{1}{t^{2k+2\mu +d-1}}\frac{\partial^\kb}{\partial x^\kb} \left(t^{2}- \|x\|^2\right)^{k+\mu-\frac12}
 \\ \times
\frac{\partial^{n-k}}{\partial t^{n-k}} \left[ t^{n+k + \b+2\mu+d-1}(1-t)^{\g+n-k} \right].
\end{multline*}
\end{defn}

\begin{thm}
The family of polynomials $\{\sU_{\kb,n}^\sJ: |\kb| \le n\}$ is a basis of $\CV_n(\VV^{d+1}, W_{\b,\g,\mu})$. 
\end{thm}

\begin{proof}
The definition of $\sU_{\kb,n}^\sJ$ can be rewritten as 
\begin{multline*}
\sU_{\kb,n}^\sJ(x,t) = \f{1}{(t^2-\|x\|^2)^{\mu - \f12}}\f{\partial^{\kb}}{\partial x^{\kb}} (t^2-\| x \|^2)^{k + \mu -\f{1}{2}} \\
 \times \f{1}{t^{2k+\b+2\mu+d-1} (1-t)^\g} \f{\partial^{n-k}}{\partial t^{n-k}} \left[ t^{n+k +\b+2\mu+d-1}(1-t)^{\g+n-k} \right],
\end{multline*}
from which it follows immediately that $\sU_{\kb,n}^\sJ$ is indeed a polynomial of degree $n$. We can prove that
it is an orthogonal polynomial similar to the case of the Laguerre polynomials. Let $P$ be a polynomial of degree $n-1$ 
in $(x,t)$ variable. Integrating by parts $k$ times with respect to $x$ gives 
\begin{align*}
 \int_{\VV^{d+1}} \sU_{\kb,n}^\sJ(x,t)& P (x,t)W_{\b,\g,\mu}(x,t) \d x \d t \\
 = \,& (-1)^k \int_{\VV^{d+1}} \frac{1}{t^{2k+2\mu +d-1}}(t^{2}-\|x\|^2)^{k+\mu -\frac{1}{2}} \\
& \times \frac{\partial^{n-k}}{\partial t^{n-k}}
\left[ (1-t)^{\gamma +n-k}t^{n+k+2\mu +\beta +d-1}\right]Q(x,t)\d x \d t,
\end{align*}
which becomes zero upon integrating by parts $n-k$ times with respect to $t$. 
\end{proof}

\begin{thm}
The polynomials $\sU^\sJ_{\kb,n}(x,t)$ and $\sV^\sJ_{\kb,n} (x,t)$ are partially biorthogonal in the
sense that, for $\kb, \jb \in \NN_0^d$ and $k = |\kb| \le n$ and $j= |\jb| \le n$, 
\begin{align} \label{eq:biortho-J}
 b_{\b,\g,\mu}^J \int_{\VV^{d+1}} & \sU^\sJ_{\kb,n}(x,t) \sV^\sJ_{\jb,n}(x,t) W_{\b,\g,\mu}(x,t) \d x \d t \\
 = & k!(n-k)! (-1)^{n} \frac{(\mu +\frac{1}{2})_k (2\mu +\b+d) _{n+k} (\g+1)_{n-k}}
 {(\mu +\frac{d+1}{2})_k (2\mu+\b+\g+d+1)_{2n}} \delta_{\kb,\jb} \notag
\end{align}
holds if $k \ge j$ or $j > k$ and $\jb- \kb$ has an odd component.
\end{thm}
 
\begin{proof}
Since $\sU^{\sJ}_{\kb,n}$ is orthogonal to polynomials of lower degree and $\sV^\sJ_{\kb,n}$ is monomial, we obtain 
\begin{align*}
\left \langle \sU^{\sJ}_{\kb,n}, \sV^{\sJ}_{\jb,n} \right \rangle 
\,& = b_{\b,\g,\mu}^{J} \int_{\VV^{d+1}} \sU^{\sJ}_{\kb,n} (x,t)\sV^{\sJ}_{\jb,n}W_{\b,\g,\mu}(x,t) \d x \d t \\
& = b_{\b,\g,\mu}^{J} \int_{\VV^{d+1}} \sU^{\sJ}_{\kb,n}(x,t) t^{n-|\jb|}x^\jb W_{\b,\g,\mu}(x,t) \d x \d t. 
\end{align*}
Integrating by parts $\kb$ times with respect to $x$, we obtain 
\begin{align*}
\left \langle \sU^{\sJ}_{\kb,n}, \sV^{\sJ}_{\jb,n} \right \rangle =\,& (-1)^k b_{\b,\g,\mu}^J
 \int_{\VV^{d+1}} \frac{t^{n-k}}{t^{2k+2\mu +d-1}}(t^{2}-\|x\|^2)^{k +\mu -\frac{1}{2}}\\
&\times \frac{\partial ^{n-k}}{\partial t^{n-k}} \left[ (1-t)^{\g+n-k} t^{n+k+2\mu +\beta +d-1}\right] \frac{\partial^\kb} {\partial x^\kb} x^{\jb}\, \d x \d t,
\end{align*}
from which the partially biorthogonal follows as in the Laguerre case.
For $\jb = \kb$, we use substitution $x=ty$ 
and integrate by parts with respect to $t$ to obtain 
\begin{align*}
\left \langle \sU^{\sJ}_{\kb,n}, \sV^{\sJ}_{\kb,n} \right \rangle 
 = \, & (-1)^k \kb! b_\mu^\BB \int_{\mathbb{B}^{d}} (1- \|y\|^2)^{k+\mu -\frac12}\d y \\
 & \times \, c_{2\mu + \b+ d-1, \g} \int_{0}^1 t^{n- k}
 \frac{\partial ^{n-k}}{\partial t^{n-k}} \left[(1-t)^{\g+n-k} t^{n+k+2\mu +\beta +d-1}\right] \d t \\
 = \, & (-1)^n \kb! \frac{b_\mu^\BB}{b_{\mu+k}^\BB} 
 (n-k)! c_{2\mu + \b+ d-1, \g} \int_{0}^1 (1-t)^{\g+n-k} t^{n+k+2\mu +\beta +d-1} \d t \\
 = \, & (-1)^n \kb! (n-k)! \frac{(\mu+\frac{1}{2}) _{k}}{(\mu +\frac{d+1}{2}) _k}
 \frac{c_{2\mu + \b+ d-1, \g}}{c_{n+k+2\mu + \b+ d-1, n-k+\g}},
\end{align*}
where $c_{\a, \g}$ is defined in \eqref{eq:bJ}, from which \eqref{eq:biortho-J} follows. 
\end{proof}

We can also define a basis in terms of $\sU_{\kb,n}^\sJ$ so that it is biorthogonal with respect to the basis of 
$\sV_{\kb,n}^\sJ$. We state the result for the case $d = 1$ as in the Laguerre case. 

\begin{prop}
For $d =1$, define the polynomials $\{\sY_{n-k,n}^\sJ: 0 \le k \le n\}$ recursively as follows: for
$k= 0, 1$, $\sY_{n-k,n}^\sJ = \sU_{n-k, n}^\sJ$, and for $k =2, 3,\ldots,n$,
\begin{equation} \label{eq:sYknJ}
 \sY_{n-k,n}^\sJ (x,t) = \sU_{n-k,n}^{\sJ}- \sum_{1 \le i \le \lfloor k/2 \rfloor} d_{i,k} \sY_{n-k+2i,n}^\sJ
\end{equation}
recursively, where the coefficients are given by 
\begin{equation*}
d_{i,k}= \frac{k! (\mu +i+n-k+1)_{i} (\gamma +k-2i+1)_{2i}}
 {2^{2i}(k-2i)! i! (\mu +n-k+1/2)_{2i} (2\mu +\beta +2n-k+1)_{2i}}.
\end{equation*}%
Then $\{\sY_{n-k,n}^\sJ: 0 \le k \le n\}$ is a basis of $\CV_n(\VV^2,W_{\b,\g,\mu})$ and it is mutually 
biorthogonal to the basis $\{\sV_{k,n}^\sJ: 0 \le k \le n\}$.
\end{prop}

The proof is parallel to the Jacobi case, we shall omit the details. 
 
\section{Generating function for $U$-family and $V$-family}
\setcounter{equation}{0}
Like the case of the unit ball, the $U$-family of orthogonal polynomials on the cone satisfies a 
generating function. For the $V$-family, however, we do not have a generating function unless we 
use the relation in Proposition \ref{prop:sS-sV} to relate $\sV_{\kb,n}$ to $\sS_{\kb,n}$ first, since then
we can obtain a generating function for $\sV_{\kb,n}$ in terms of a generating function for $\sS_{\kb,n}$. 
We shall give the generating function for $\sV_{\kb,n}$ in this section as well. In both cases, the 
generating functions are given for the Laguerre polynomials and the Jacobi polynomials on the cone.

\subsection{Laguerre polynomials on the cone} 
The polynomials $U_{\kb,n}^\sL$ are indexed by $(\kb,n)$ with $\kb \in \NN_0^d$, $|\kb| \le n$ and 
$n \in \NN_0$. Its generating function is given below. 

\begin{prop}
For $\bb \in \BB^d$, $r> 0$ and $r < 1$, the generating function for $\sU_{\kb,n}^\sL$ is given by
\begin{align} \label{eq:generatUL}
 \sum_{n=0}^\infty & \sum_{|\kb| \le n} \frac{(-1)^{k} (2\mu)_{|\kb|}}{ 2^{|\kb|} (\mu+\f1 2)_{|\kb|}|\kb|!(n-|\kb|)!}
 \sU_{\kb,n}^\sL(x,t) \bb^\kb r^n \\
 & =\frac{(1-r)^{2\mu} \e^{-\frac{t r}{1-r}}}
 { (1-r)^{\b+d} \big( [(1-r)^2 - r \la \bb, x \ra]^2 + r^2 \|\bb\|^2 (t^2 - \|x\|^2) \big)^\mu }. \notag
\end{align}
\end{prop}

\begin{proof}
Changing variable $x = t y$, we can write 
$$
\frac{1}{(t^2 -\|x\|^2)^{\mu-\f12} }
\frac{\partial^{\kb}}{\partial x^\kb}(t^2 -\|x\|^2)^{k+\mu-\f12} = 
\frac{t^k }{(1 -\|y\|^2)^{\mu-\f12}} \f{\partial^{\kb}}{\partial y^\kb}(1 -\|y\|^2)^{k+\mu-\f12}.
$$
Consequently, using the Rodrigues formulas for $U_\kb$ in \eqref{eq:RodrigB} and the Laguerre polynomial, 
we see that the polynomial $\sU_{\kb,n}^\sL$ can be written as 
$$
 \sU_{k,n}^\sL(x,t) = \frac{2^{|\kb|} (\mu+\f1 2)_{|\kb|} !\kb|! (n-|\kb|)!}{(-1)^k (2\mu)_{|\kb|}}
 t^k U_{\kb}\left( \frac{x}{t} \right) L_{n-k}^{2k + \a}(t), 
$$
where $\a = 2\mu + \b+ d-1$. Hence, using the generating formula for the Laguerre polynomials first, we see that the 
left-hand side of \eqref{eq:generatUL} is equal to 
\begin{align*}
 \mathrm{LHS} \, & = \sum_{n=0}^\infty \sum_{k=0}^n \sum_{|\kb|= k} t^k U_\kb\left( \frac{x}{t} \right)
 L_{n-k}^{2k + \a}(t) \bb^\kb r^n \\
 & = \sum_{k=0}^\infty \sum_{|\kb|= k} t^k U_\kb \left( \frac{x}{t} \right) \bb^\kb 
 \sum_{n=k}^\infty L_{n-k}^{2k + \a}(t) r^n \\
 & = \sum_{k=0}^\infty \sum_{|\kb|= k} t^k U_\kb \left( \frac{x}{t} \right) \bb^\kb r^k 
 \frac{1}{(1-r)^{2k+\a+1}} \e^{\frac{- rt}{1-r}}.
\end{align*}
Now, applying the generating formula for $U_\kb$ in \eqref{eq:generatingB} with $\ab = t r |\bb|/(1-r)^2$, we obtain
\begin{align*}
 \mathrm{LHS} = \frac{1}{ \big( ( 1- \frac{r}{(1-r)^2} \la \bb,x \ra)^2 + r^2 \frac{\|\bb\|^2}{(1-r)^4} (t^2 - \|x\|^2) \big)^\mu}
 \frac{1}{(1-r)^{\a+1}} \e^{\frac{- rt}{1-r}},
\end{align*}
which becomes \eqref{eq:generatUL} after rearranging the terms. 
\end{proof}

We now give a generating function for $\sS_{\kb,n}^\sL$ given by 
\begin{align*} 
 \sS_{\kb,n}^\sL (x,t):= \wh L_{n-k}^{2k+2\mu+\b+d-1}(t) t^{k} V_\kb \left(\f{x}{t} \right), \quad |\kb| = k, \,\,
 \kb \in \NN_0^d, \,\, 0\le k \le n. 
\end{align*}
Using the generating function \eqref{eq:generatingBV} in place of~\eqref{eq:generatingB}, the generating function
can be derived exactly as in the proof for \eqref{eq:generatUL}. The result is the following.

\begin{prop}
For $\bb \in \BB^d$, $r> 0$ and $r < 1$, the generating function for $\sS_{\kb,n}^\sL$ is given by
\begin{align*} 
 \sum_{n=0}^\infty & \sum_{|\kb \le n} \frac{(-1)^{n-|\kb|} 2^{|\kb|} (\mu+\f{d-1}2)_{|\kb|}} { (n-|\kb|)! \kb!}
 \sS_{\kb,n}^\sL(x,t) \bb^\kb r^n \\
 & =\frac{(1-r)^{2\mu+d} \e^{-\frac{t r}{1-r}}}
 { (1-r)^{\b+2} \big( (1-r)^4 - 2 r (1-r)^2 \la \bb, x \ra +t^2 r^2 \|\bb\|^2 \big)^{\mu+\frac{d-1}2}}. \notag
\end{align*}
\end{prop}

\subsection{Jacobi polynomials on the cone} 
Similarly to the Laguerre case, the generating function for the $U$-family of the Jacobi polynomials on the cone 
is given as follows. 

\begin{prop}
For $\bb \in \BB^d$, $r> 0$ and $r < 1$, the generating function for $\sU_{\kb,n}^\sJ$ is given by
\begin{align} \label{eq:generatUJ}
 \sum_{n=0}^\infty & \sum_{|\kb| \le n} \frac{(-1)^k (2\mu)_{|\kb|}}{ 2^{|\kb|} (\mu+\f1 2)_{|\kb|} |\kb|!(n-|\kb|)!}
 \sU_{\kb,n}^\sJ(x,t) \bb^\kb r^n \\
 & = \frac{2^{2\mu+\b+\g+d-1} R^{-1} (1+r+R)^{-\g} (1-r+R)^{2\mu - \b - d +1}} 
 { \big( [(1-r+R)^2 - 4 r \la \bb, x \ra]^2 + 16 r^2 \|\bb\|^2 (t^2 - \|x\|^2) \big)^\mu }, \notag
\end{align}
where 
$$
 R =\left[ (1-r)^2 + 4 r t\right ] ^{\frac12}. 
$$

\end{prop}

\begin{proof}
Changing variable $x = t y$, we can rewrite the polynomial $\sU_{\kb,n}^J$ as 
$$
 \sU_{k,n}^\sJ(x,t) = \frac{2^{|\kb|} (\mu+\f1 2)_{|\kb|} !\kb|! (n-|\kb|)!}{(-1)^k (2\mu)_{|\kb|}}
 t^k U_{\kb}\left( \frac{x}{t} \right) P_{n-k}^{(2k + \a, \g)}(1-2t), 
$$
where $\a = 2\mu + \b+ d-1$. We now use the generating formula for the Jacobi polynomials, given by
$$
 \sum_{n=0}^\infty P_n^{(\a,\b)}(1-2t) r^n = 2^{\a+\b} R^{-1} (1-r+R)^{-\a} (1+r+R)^{-\b}, 
$$
with $R$ defined as in the statement, which shows that the left-hand side of \eqref{eq:generatUJ} is equal to 
\begin{align*}
 \mathrm{LHS} \, & = \sum_{n=0}^\infty \sum_{k=0}^n \sum_{|\kb|= k} t^k U_\kb\left( \frac{x}{t} \right)
 P_{n-k}^{(2k + \a,\g)}(1-2t) \bb^\kb r^n \\
 & = \sum_{k=0}^\infty \sum_{|\kb|= k} t^k U_\kb \left( \frac{x}{t} \right) \bb^\kb 
 \sum_{n=k}^\infty P_{n-k}^{(2k + \a,\g)} (1-2t) r^n\\ 
 & = \sum_{k=0}^\infty \sum_{|\kb|= k} t^k U_\kb \left( \frac{x}{t} \right) \bb^\kb r^k 
 2^{2k+\a+\g} R^{-1} (1-r+R)^{- 2k-\a} (1+r+R)^{-\g} .
\end{align*}
Now, we apply the generating formula for $U_\kb$ in \eqref{eq:generatingB} with $\ab = 4 t r |\bb|/(1-r+R)^2$
and simplify the result to complete the proof. 
\end{proof}
 
Similarly, as in the Laguerre case, we can give a generating function for $\sS_{\kb,n}^\sJ$, 
\begin{align} \label{eq:S-OPconeJ}
 \sS_{\kb,n}^\sJ (x,t):= \wh P_{n-k}^{(2k+2\mu+\b+d-1,\g)}(1-2t) t^{k} V_\kb \left(\f{x}{t} \right), 
\end{align}
for $|\kb| = k$, $\kb \in \NN_0^d$ and $0\le k \le n$. Using the generating function \eqref{eq:generatingBV} 
in place of~\eqref{eq:generatingB}, the proof follows exactly as that of \eqref{eq:generatUL}. The result 
is the following. 

\begin{prop}
For $\bb \in \BB^d$, $r> 0$ and $r < 1$, the generating function for $\sS_{\kb,n}^\sJ$ is given by
\begin{align*} 
 \sum_{n=0}^\infty \sum_{|\kb| \le n} & \frac{(-1)^{n-k} 2^{|\kb|} (\mu+\f{d-1} 2)_{|\kb|} (n+k+2\mu+\b+\g+1)_{n-k}}
 {(n-|\kb|)! \kb!}
 \sS_{\kb,n}^\sJ(x,t) \bb^\kb r^n \\
 & = \frac{2^{2\mu+\b+\g+d-1} R^{-1} (1+r+R)^{-\g} (1-r+R)^{2\mu +d - \b -1}} 
 { \big( (1-r+R)^4 - 8 r (1-r+R)^2 \la \bb, x \ra + 16 r^2 t^2 \|\bb\|^2 \big)^{\mu +\frac{d-1}{2} }}, \notag
\end{align*}
where $R$ is the same as in \eqref{eq:generatUJ}.
\end{prop}

\end{document}